\documentclass[11pt]{article}
\usepackage{amssymb}
\usepackage{amsthm}
\usepackage{amsmath}
\usepackage{fullpage,algorithmic}
\usepackage{graphicx,url,amsmath,amsthm,amsfonts,amssymb,subfigure,bbm,enumerate}
\usepackage{tikz}
\definecolor{darkgreen}{rgb}{0.1,0.6,0.1}
\usepackage[pagebackref,letterpaper=true,colorlinks=true,pdfpagemode=none,linkcolor=blue,citecolor=darkgreen,pdfstartview=FitH]{hyperref}
\usepackage{pdfsync}

\newtheorem{theorem}{Theorem}[section]
\newtheorem{lemma}[theorem]{Lemma}

\newtheorem{claim}[theorem]{Claim}

\newtheorem{corollary}[theorem]{Corollary}

\newtheorem{remark}{Remark}[section]

\newtheorem{algorithm}{Algorithm}

\newcommand{\pr}{\mathbb{P}}

\newcommand{\jnote}[1]{}

\newcommand{\E}{{\mathbb E}}
\newcommand{\diam}{\mathrm{diam}}

\newcommand{\supp}{\mathrm{supp}}
\newcommand{\remove}[1]{}
\newcommand{\1}{\mathbf{1}}

\newcommand{\e}{\varepsilon}
\newcommand{\eps}{\varepsilon}

\newcommand{\poly}{\textup{poly}}
\newcommand{\R}{\mathbb{R}}
\newcommand{\norm}[1]{\left\|#1\right\|}

\DeclareMathOperator{\argmin}{argmin}
\newcommand*{\defeq}{\mathrel{\vcenter{\baselineskip0.5ex \lineskiplimit0pt
                     \hbox{\scriptsize.}\hbox{\scriptsize.}}}%
                     =}

\newcommand*{\defeqb}{=\mathrel{\vcenter{\baselineskip0.5ex \lineskiplimit0pt
                     \hbox{\scriptsize.}\hbox{\scriptsize.}}}%
                     }

\begin{document}

\title{{\bf Multi-way spectral partitioning and higher-order \\ Cheeger inequalities}}
\author{James R. Lee\thanks{Department of Computer Science \& Engineering, University of Washington.  Partially supported by NSF grants CCF-0644037, CCF-0915251, and a Sloan Research
Fellowship. Email:\protect\url{jrl@cs.washington.edu}. } \and Shayan Oveis Gharan\thanks{Department of Management Science and Engineering, Stanford University.  Supported by a
Stanford Graduate Research Fellowship. Part of this work was done while the author was a summer intern at Microsoft Research New England. Email:\protect\url{shayan@stanford.edu}.} \and
Luca Trevisan\thanks{Department of Computer Science, Stanford University.  This material is based
on work supported by the National Science Foundation under grant CCF-1017403. Email:\protect\url{trevisan@stanford.edu}.}}
\date{}

\maketitle

\begin{abstract}
A basic fact in spectral graph theory is that the number of connected components in an undirected graph is equal to the multiplicity of the eigenvalue zero in the Laplacian matrix of
the graph. In particular, the graph is disconnected if and only if there are at least two eigenvalues equal to zero.
Cheeger's inequality and its variants
provide an approximate version of the latter fact; they state that a graph has a sparse cut if and only if there are at least two eigenvalues that are close to zero.

It has been conjectured that an analogous characterization holds for higher multiplicities:  There are $k$ eigenvalues close to zero if and only if the vertex set can be
partitioned into $k$ subsets, each defining a sparse cut. We resolve this conjecture positively. Our result provides a theoretical justification for clustering algorithms that use the
bottom $k$ eigenvectors to embed the vertices into $\mathbb R^k$, and then apply geometric considerations to the embedding.

We also show that these techniques yield a nearly optimal quantitative connection between the expansion of sets of size $\approx n/k$ and $\lambda_k$, the $k$th smallest eigenvalue of
the normalized
Laplacian, where $n$ is the number of vertices.
In particular, we show that in every graph there are at least $k/2$ disjoint sets (one of which will have size at most $2n/k$), each having  expansion at most $O(\sqrt{\lambda_k \log
k})$.   Louis, Raghavendra, Tetali, and Vempala have independently proved
a slightly weaker version of this last result.
The $\sqrt{\log k}$ bound is tight,
up to constant factors, for the ``noisy hypercube'' graphs.
\end{abstract}

\newpage

\setcounter{tocdepth}{2} \tableofcontents

\newpage

\section{Introduction}

Let $G=(V,E)$ be an undirected, $d$-regular graph. Its normalized Laplacian matrix $L \in \mathbb R^{V \times V}$ is given
by $L = I - \frac{1}{d} A$, where $A$ is the adjacency matrix of $G$.  For the moment, we
confine ourselves to unweighted, regular graphs, while the results in the paper are presented
for arbitrary weighted graphs, with suitable changes to $L$.
It is easy to see that $L$ is a positive semi-definite matrix, and its eigenvalues satisfy
$0 = \lambda_1 \leq \lambda_2 \leq \cdots \leq \lambda_{|V|}$.
Elementary arguments show that
the number of connected components of $G$ is precisely the
multiplicity of the eigenvalue zero, that is, $\lambda_k = 0$ if and only if the graph has at least $k$ connected components.

Cheeger's inequality for graphs \cite{AM85,Alon86,SinclairJerrum} yields
a robust version of this fact for $k=2$.
To state it, we introduce some notation.
For any subset $S \subseteq V$, define the {\em expansion} of $S$
to be the quantity
$$
\phi_G(S) = \frac{|E(S,\overline{S})|}{d \,|S|}\,,
$$
where $E(S,\overline{S})$ denotes the set of edges of $G$ crossing from $S$ to its complement.
We may also define, for every $k \in \mathbb N$, the
{\em $k$-way expansion constant,}
$$
\rho_G(k) = \min_{S_1, S_2, \ldots, S_k} \max \{ \phi_G(S_i) : i=1,2,\ldots,k\},
$$
where the minimum is over all collections of $k$ non-empty, disjoint subsets $S_1, S_2, \ldots, S_k \subseteq V$.
It is an easily verifiable fact that $\rho_G(k)=0$ if and only if $\lambda_k=0$.
Cheeger's inequality offers the following quantitative
connection betwen $\rho_G(2)$ and $\lambda_2$,
\begin{equation}\label{eq:introcheeger}
\frac{\lambda_2}{2} \leq \rho_G(2) \leq \sqrt{2 \lambda_2}\,.
\end{equation}
We remark that the left-hand side follows easily, and the non-trivial
content of the connection is contained in the right-hand side inequality.

The discrete version of
Cheeger's inequality is proved via a simple spectral partitioning algorithm.
Besides being an important theoretical tool,
since their inception spectral methods have been
  used for solving a wide range of optimization problems, from graph
  coloring \cite{AspvallGilbert,AlonKahale} to image segmentation
  \cite{ShiMalik,TolliverMiller} to web search \cite{Kleinberg,BrinPage}.

\medskip
\noindent
{\bf Higher-order Cheeger inequalities.}
In general, we study higher-order analogs of \eqref{eq:introcheeger},
and develop new multi-way spectral partitioning algorithms.
A special case of one of our main theorems (see Section \ref{sec:higherorder} and Theorem \ref{thm:fewpiecesbetter})
follows.  It offers a strong quantitative version of the fact that $\rho_G(k)=0 \iff \lambda_k=0$.

\begin{theorem}\label{thm:kway}
For every graph $G$, and every $k \in \mathbb N$, we have
\begin{equation}\label{eq:kway}
\frac{\lambda_k}{2} \leq \rho_G(k) \leq O(k^2) \sqrt{\lambda_k}\,.
\end{equation}
\end{theorem}

This resolves a conjecture of
Miclo \cite{Miclo2008}; see also \cite{DJL12}, where some special
cases are considered.
Moreover, Miclo \cite{Miclo13} has used Theorem \ref{thm:kway} as the key step in establishing a 40-year-old conjecture
of Simon and H$\o$egh-Krohn \cite{SH72}.  We discuss
this connection briefly at the end of the present section.

We remark that from Theorem \ref{thm:kway}, it is easy
to find a {\em partition} of the vertex set into $k$ non-empty pieces such that
every piece in the partition has expansion $O(k^3) \sqrt{\lambda_k}$
(see Theorem \ref{thm:kpartition}).
It is known that a dependence on $k$ in the right-hand side of \eqref{eq:kway}
is necessary; see Section \ref{sec:noisycube}.

\medskip

Moreover, our proof is algorithmic
and leads to new algorithms for $k$-way spectral partitioning.
This provides a theoretical justification for clustering algorithms that use the
bottom $k$ eigenvectors of the Laplacian\footnote{Equivalently, algorithms
that use the top $k$ eigenvectors of the adjacency matrix.} to embed the vertices into $\mathbb R^k$, and then apply geometric considerations to the embedding.
See \cite{VM03} for a survey of such approaches.
As a particular example, consider the work of Jordan, Ng and Weiss \cite{NJW02} which applies
a $k$-means clustering algorithm to the embedding in order to achieve a $k$-way partitioning.
Our proof of Theorem \ref{thm:kway} employs a similar algorithm, where the $k$-means step is replaced by
a random geometric partitioning.  It remains an interesting
open problem whether $k$-means itself can be analyzed in this setting.
See the discussion at the end of Section \ref{sec:gaps} for some
results in this direction.

\medskip
\noindent
{\bf Finding many sets and small-set expansion.}
If one is interested in finding slightly fewer sets, our approach performs significantly better.

\begin{theorem}\label{thm:fewer}
For every graph $G$, and every $k \in \mathbb N$, we have
\begin{equation}\label{eq:neartight}
\rho_G(k) \leq O(\sqrt{\lambda_{2k} \log k})\,.
\end{equation}
If $G$ is planar then, the bound improves to,
\begin{equation}\label{eq:planarintro}
\rho_G(k) \leq O(\sqrt{\lambda_{2k}})\,.
\end{equation}
More generally, if $G$ excludes $K_h$ as a minor, then
$$
\rho_G(k) \leq O(h^2 \sqrt{\lambda_{2k}})\,.
$$
\end{theorem}

We remark that the bound \eqref{eq:neartight} holds with $2k$ replaced by $(1+\delta)k$ for any $\delta > 0$,
but where the leading constant now becomes $\delta^{-3}$; see Corollary \ref{cor:optfewsets}.
Louis, Raghavendra, Tetali and Vempala \cite{LRTV12}
have independently proved a somewhat weaker version of the bound \eqref{eq:neartight},
using rather different techniques.
Specifically, they show that there exists an absolute constant $C > 1$
such that $\rho_G(k) \leq O(\sqrt{\lambda_{Ck} \log k})$.

In particular, Theorem \ref{thm:fewer} has applications to the small-set expansion problem
in graphs, which is fundamentally connected to the Unique Games Conjecture
and many other problems in approximation algorithms (see \cite{RS10, RST10}).
To capture the expansion of small sets in graphs, we define the value,
$$
\varphi_G(k) = \min_{S \leq |V|/k} \phi_G(S)\,.
$$
Clearly $\varphi_G(k) \leq \rho_G(k)$ for every $k \in \mathbb N$.

\medskip

Arora, Barak and Steurer \cite{ABS10} prove the bound,
$$\varphi_G(k^{1/100}) \leq O(\sqrt{\lambda_{k} \log_k n}),$$
where $n=|V|$.  Note that for $k = n^{\varepsilon}$ and $\epsilon \in (0,1)$,
one achieves an upper bound of $O(\sqrt{\lambda_k})$,
and this small loss in the expansion constant
is crucial for applications to approximating small-set expansion.
This was improved further in Steurer's thesis \cite{SteurerThesis} by
showing that for every $\alpha > 0$,
$$
\varphi_G(k^{1-\alpha})\leq O(\sqrt{(\lambda_{k}/\alpha) \log_k n})\,.
$$
Such a bound is also obtained in the works \cite{OT12,OW12}.
These bounds work fairly well for large values of $k$,
but give less satisfactory results when $k$ is smaller.

\medskip

Louis, Raghavendra, Tetali and Vempala \cite{LRTV11}
proved that
$$
\varphi_G(\sqrt{k}) \leq O(\sqrt{\lambda_k \log k}),
$$
and conjectured that $\sqrt{k}$ could be replaced by $k$.
Theorem \ref{thm:fewer} immediately yields,
\begin{equation}\label{eq:almost}
\varphi_G(k/2) \leq O(\sqrt{\lambda_k \log k})
\end{equation}
resolving their conjecture up to a factor of 2 (and actually,
as discussed earlier, up to a factor of $1+\delta$ for every $\delta > 0$).

Moreover, \eqref{eq:almost} is quantitatively optimal for the noisy hypercube graphs (see Section \ref{sec:noisycube}),
yielding an optimal connection between the $k$th Laplacian eigenvalue and expansion of sets of size $\approx n/k$.

It is interesting to note that in \cite{KLPT11}, it is shown that for $n$-vertex, bounded-degree planar
graphs, one has $\lambda_k = O(k/n)$.
Thus the spectral algorithm guaranteeing \eqref{eq:planarintro} partitions such a planar graph
into $k$ disjoint pieces, each of expansion $O(\sqrt{k/n})$.
This is tight, up to a constant factor, as one can easily see for an $\sqrt{n} \times \sqrt{n}$ planar grid, in
which case the set of size $\approx n/k$ with minimal expansion is a $\sqrt{n/k} \times \sqrt{n/k}$ subgrid.

\medskip
\noindent
{\bf Large gaps in the spectrum.}  We recall that in the practice of spectral clustering,
it is often observed that the correct number of clusters is indicated by
a large gap between adjacent eigenvalues, i.e., if $\lambda_{k+1} \gg \lambda_k$,
then one expects the input graph can be more easily partitioned
into $k$ pieces than $k+1$.  In Section \ref{sec:gaps}, we prove a result supporting this phenomenon.

\begin{theorem}
There is a constant $C > 0$ such that
for every graph $G$ and $k \in \mathbb N$, the following holds.
If $\lambda_{4k} \geq C (\log k)^2 \lambda_{2k}$, then
$$
\rho_G(k) \leq O(\sqrt{\lambda_{2k}})\,.
$$
\end{theorem}

The key point is that the implicit constant in the upper bound is independent of $k$, unlike
the bound \eqref{eq:neartight}.

\medskip
\noindent
{\bf The relation to hyperboundedness and spectral gaps of Markov operators.}
Consider a probability space $(\Omega,\mathcal F,\mu)$.  A self-adjoint operator
$M : L^2(\mu) \to L^2(\mu)$ is said to be {\em Markovian} if, whenever $f \in L^2(\mu)$, we have
$f \geq 0 \implies Mf \geq 0$ and $M \1 = \1$.
One says that $M$ is {\em ergodic} if $Mf=f$ implies that $f$ is a multiple of $\1$.

Such an operator $M$ may not have any eigenvectors other than $\1$, but one defines
its spectrum $\sigma(M)$ to be the set of $\lambda \in [-1,1]$ such that $\lambda I-M$ fails
to be invertible.  An ergodic Markov operator $M$ is said to have a {\em spectral gap} if
there is a $\delta > 0$ such that $\sigma(M) \subseteq \{1\} \cup [-1,1-\delta]$.
Finally, say that $M$ is {\em hyperbounded} if there exists a $p > 2$ such that
$$
\|M\|_{L^2(\mu) \to L^p(\mu)} < \infty\,.
$$
In \cite{Miclo13}, the following theorem is proved.

\begin{theorem}[Miclo]
\label{thm:miclonew}
If a self-adjoint, ergodic Markov operator is hyperbounded, then it has a spectral gap.
\end{theorem}

This was conjectured by Simon and H$\o$egh-Krohn \cite{SH72} for the special case of
Markov semi-groups.  They actually indicated that the conjecture was probably
false even in this specialized setting.
Miclo uses Theorem \ref{thm:kway} as a fundamental step in the proof of Theorem \ref{thm:miclonew}.
The basic idea is to relate the operator $2 \rightarrow p$ norm to expansion
of small sets in a graph (or, more generally, in the underlying probability space $(\Omega,\mu)$).
Then one uses Theorem \ref{thm:kway} to relate expansion of small sets to the spectrum of the operator.
One can consult \cite{BBHKSZ12} for a detailed discussion of operator norms and small-set expansion
from a computational perspective.

\subsection{High-dimensional spectral partitioning}

We now present an overview of the proofs of our main theorems, as well
as explain our general approach to multi-way spectral partitioning.
Let $G=(V,E)$ be an undirected, $d$-regular graph.
To begin, for any $f : V \to \ell_2$, we recall the {\em Rayleigh quotient},
$$
\mathcal R_G(f) \defeq \frac{\sum_{\{u,v\} \in E} \|f(u)-f(v)\|^2}{d \sum_{u \in V} \|f(u)\|^2}\,.
$$

Cheeger's inequality (see Lemma \ref{lem:cheeger})  proves
that for any $f~:~V \to \ell_2$, it is possible to find a subset $S \subseteq \{ v \in V : f(v) \neq 0\}$ such that
that $\phi_G(S) \leq \sqrt{2 \mathcal R_G(f)}$.
Thus in order to find $k$ disjoint, non-expanding subsets $S_1, S_2, \ldots, S_k \subseteq V$,
it suffices to find $k$ disjointly supported functions $\psi_1, \psi_2, \ldots, \psi_k : V \to \ell_2$
such that $\mathcal R_G(\psi_i)$ is small for each $i =1,2,\ldots, k$.

In fact, in the same paper that Miclo conjectured the validity of Theorem \ref{thm:kway}, he
conjectured that finding such a family $\{\psi_i\}$ should be possible \cite{Miclo2008,DJL12}.  We resolve this conjecture
and prove the following theorem in Section \ref{sec:higherorder}.

\begin{theorem}\label{thm:miclo}
For any graph $G=(V,E)$ and any $k \in \mathbb N$, there exist disjointly supported functions $\psi_1, \psi_2, \ldots, \psi_k : V \to \mathbb R$
such that for each $i=1,2,\ldots, k$, we have
$$
\mathcal R_G(\psi_i) \leq O(k^6)\, \lambda_k\,.
$$
\end{theorem}

To prove this, we start with an orthonormal system of eigenfunctions of the Laplacian, $$f_1, f_2, \ldots, f_k : V \to \mathbb R\,,$$
where $f_i$ has eigenvalue $\lambda_i$.  We then construct the embedding $F : V \to \mathbb R^k$
given by \begin{equation}\label{eq:specembedding}
F(v) = (f_1(v), f_2(v), \ldots, f_k(v))\,.\end{equation}
Observe that $\mathcal R_G(F) \leq \lambda_k$.

Thus our goal is now to ``localize'' $F$ on $k$ disjoint regions
to produce disjointly supported functions $\psi_1, \psi_2, \ldots, \psi_k : V \to \mathbb R^k$, each
with small Rayleigh quotient.
(It is elementary to see that for any map $\psi : V \to \mathbb R^k$, there exists
some coordinate $j \in \{1,2,\ldots,k\}$ such that the $\mathbb R$-valued map $\tilde \psi(v) = \psi(v)_j$ has
$\mathcal R_G(\tilde \psi) \leq \mathcal R_G(\psi)$.)
In order to ensure that $\mathcal R_G(\psi_i)$
is small for each $i$, we must ensure that each region captures a large fraction of the $\ell^2$ mass of $F$,
and that our localization process is sufficiently smooth.

\medskip
\noindent
{\bf Isotropy and spreading.}
The first problem we face is that, in order to find $k$ disjoint regions each with large $\ell^2$ mass, it should be that
the $\ell^2$ mass of $F$ is sufficiently well-spread.
This follows from the following {\em isotropy} property of $F$ (see Lemma
\ref{lem:spread}):  For any vector $x \in S^{k-1}$  (the unit sphere of $\mathbb R^k$),
\begin{equation}\label{eq:isotropy}
\sum_{v \in V} \langle x, F(v)\rangle^2 = 1\,.
\end{equation}
On the other hand, it straightforward to check that,
$$
\sum_{v \in V} \|F(v)\|^2 = k\,,
$$
thus it is impossible for the $\ell^2$ mass of $F$ to ``concentrate'' along fewer than
$k$ directions $x_1, x_2, \ldots, x_k \in S^{k-1}$.

A natural approach would be to find (at least) $k$ such directions, and then define,
$$
\psi_i(v) = \begin{cases}
F(v) & \textrm{if $F(v)$ has large projection on $x_i$} \\
0 & \textrm{otherwise.}
\end{cases}
$$
Unfortunately, this sharp cutoff could make the value $$\sum_{\{u,v\} \in E} \|\psi_i(u)-\psi_i(v)\|^2,$$
much larger than the corresponding quantity for $F$.  Thus we must pursue a smoother approach
for localizing $F$.

\begin{figure}
\begin{center}
\begin{tikzpicture}[inner sep=1.2pt,scale=.85,pre/.style={<-,shorten <=2pt,>=stealth,thick}, post/.style={->,shorten >=1pt,>=stealth,thick}]
\draw[-] (-3,0) -- (3.5,0);
\draw[->,line width=1.2pt] (3.35,0) -- (3.5,0);
\draw[-] (0,-3) -- (0,3.5);
\draw[->,line width=1.2pt] (0,3.35) -- (0,3.5);
\path [-] (3,0) edge node [below] {$f_2$} (3.5,0);
\path [-] (0,3) edge node [above left] {$f_3$} (0,3.5);
\tikzstyle{every node} = [draw, circle,fill,black];

\foreach \r/\a in {1/317, 1.5/310, 2/322, 2.5/315,3/320} {
\path  (0:0)+(\a:\r) node [fill=black]  (a_\r) {};
}
\foreach \r/\a in {1/93, 1.5/82, 2/87, 2.5/93,3/97} {
\path  (0:0)+(\a:\r) node [fill=black]  (a_\r) {}; }
\foreach \r/\a in {1/217, 1.5/230, 2/225, 2.5/233,3/220} {
\path  (0:0)+(\a:\r) node [fill=black]  (a_\r) {}; }

\foreach \ang in {80, 100, 215,235,305,325}{
\draw[dashed,line width=1.2pt] (0,0) -- (\ang:3.5cm);
}


\end{tikzpicture}
\caption{Partitioning according to the radial distance.\label{fig:shayan}}
\end{center}
\end{figure}
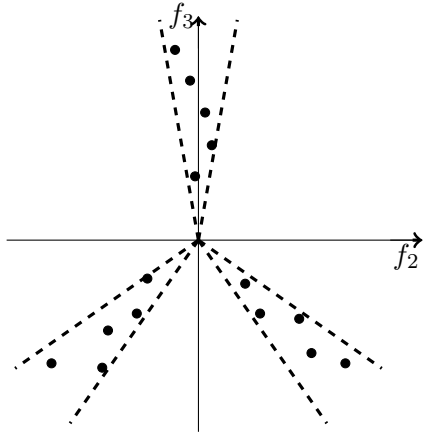

\medskip
\noindent
{\bf The radial projection distance.}
Our method of smooth localization depends crucially on defining a proper notion of
distance between vertices, based on the map $F$.    We would like to think of two vertices $u,v \in V$
as close if their Euclidean distance $\|F(u)-F(v)\|$ is small compared to their norms $\|F(u)\|, \|F(v)\|$.
To capture this, we define the {\em radial projection distance} via,
$$
d_F(u,v) = \left\|\frac{F(u)}{\|F(u)\|} - \frac{F(v)}{\|F(v)\|}\right\|\,.
$$
Note that a ball in $d_F$ corresponds to a cone in $\mathbb R^k$; see Figure \ref{fig:shayan}.

Our goal now becomes to find separated regions $S_1, \ldots, S_k \subseteq V$ in $d_F$, each of which contains a large fraction of
the $\ell^2$ mass of $F$.  If these regions are far enough apart, then there is a way to allow $\psi_i$
to degrade gracefully off of $S_i$, ensuring that $\mathcal R_G(\psi_i)$ remains small; see Lemma \ref{lem:bump}.

The isotropy condition \eqref{eq:isotropy} gives us the following spreading property of $d_F$:
If $S \subseteq V$, then
\begin{equation}\label{eq:introspread}
\diam(S, d_F) \leq \frac{1}{2} \implies \sum_{v \in S} \|F(v)\|^2 \leq \frac{2}{k} \sum_{v \in V} \|F(v)\|^2\,.
\end{equation}
In other words, sets of small $d_F$-diameter cannot contain a large fraction of the $\ell^2$ mass.
This will be essential in finding regions $\{S_i\}$.

\medskip
\noindent
{\bf Finding separated regions:  Random space partitions.}
In order to find many separated regions,
we rely on the theory of random partitions discussed in Section \ref{sec:rps}.
Roughly speaking, this partitions $\mathbb R^k$ (and thus our set of points)
randomly into pieces of diameter at most $1/2$ so that the expected fraction of $\ell^2$ mass
which is close to the boundary of the partition is small.  Thus we can take unions
of the interiors of the pieces to find separated sets.  Furthermore, no set
in the partition can contain a large fraction of the $\ell^2$ mass, due to the
spreading property of $d_F$ \eqref{eq:introspread}.
This is carried out in Section \ref{sec:randompart}.
We use these separated sets as the supports of our family $\{\psi_i\}$,
allowing us to complete the proof of Theorem \ref{thm:miclo}.
Our use of random partitions to construct disjoint bump functions
is similar to the approach in \cite{BLR08,KLPT11}.

The notion of ``close to the boundary'' depends on the dimension $k$, and thus the smoothness of
our maps $\{\psi_i\}$ will degrade as the dimension grows.
For many families of graphs, however, we can appeal
to special properties of their intrinsic geometry.

\medskip
\noindent
{\bf Exploiting the intrinsic geometry.}
It is well-known that the shortest-path metric on a planar graph
has many nice properties, but $d_F$ is, in general, not a shortest-path geometry.
Thus it is initially unclear how one might prove a bound like \eqref{eq:planarintro}
using our approach.  The answer is to combine information from
the spectral embedding with the intrinsic geometry of the graph.

We define $\hat d_F$ as the shortest-path pseudometric on $G$,
where the length of an edge $\{u,v\} \in E$ is precisely $d_F(u,v)$.
In Sections \ref{sec:smoothlocal} and \ref{sec:randompart},
we show that it is possible to do the partitioning in the metric $\hat d_F$,
and thus for planar graphs (and other generalizations), we are able
to achieve dimension-independent bounds in Theorem \ref{thm:fewer}.

This technique also addresses a common shortcoming of spectral methods:
The spectral embedding can lose auxiliary information about the input data
that could help with clustering.  Our ``hybrid'' technique for planar graphs
suggests that such information (in this case, planarity) can be fruitfully combined with
the spectral computations.

\medskip
\noindent
{\bf Dimension reduction.}
In order to obtain
the tight bound \eqref{eq:neartight} for general graphs, we have to improve
the quantitative parameters of our construction.
The main loss in our preceding construction comes from the ambient dimension $k$.

Thus our first step is to apply dimension-reduction techniques:
We randomly project our points from $\mathbb R^k$ into $\mathbb R^{O(\log k)}$.
Let $F' : V \to \mathbb R^{O(\log k)}$ be the resulting map.
While it is easy to see that $\mathcal R_G(F') \asymp \mathcal R_G(F)$ with
high probability, it is not,
a priori, clear why $O(\log k)$ dimensions suffices for maintaining the
spreading properties of $F$.  Indeed, the isotropy condition \eqref{eq:isotropy}
will generally fail for $F'$.
Although the proof is delicate (see Lemma \ref{lem:dimreduce}),
the basic idea is this:  If $d_F$ satisfies \eqref{eq:introspread},
but $d_{F'}$ fails to satisfy a related property, then a $\gg \frac{1}{k}$
fraction of the $\ell^2$ mass has to have moved significantly in the dimension reduction step,
and such an event is unlikely for a random mapping into $O(\log k)$ dimensions.

\medskip
\noindent
{\bf A new multi-way Cheeger inequality.}
Dimension reduction only yields a loss of $O(\log k)$ in \eqref{eq:neartight}.
In order to get the bound down to $\sqrt{\log k}$, we abandon
our goal of localizing eigenfunctions.  In Section \ref{sec:mwcheeger},
we give a new multi-way Cheeger rounding algorithm that combines
random partitions of the radial projection distance $d_F$,
and random thresholding based on $\|F(\cdot)\|$ (as in Cheeger's inequality).
By analyzing these two processes simultaneously, we are able to achieve
\eqref{eq:neartight}.
In addition, we use this method to achieve the stated bound in \eqref{eq:kway}.

\subsection{A general algorithm}
\label{sec:genalg}

Given a graph $G=(V,E)$ and any embedding $F : V \to \mathbb R^k$
(in particular, the spectral embedding \eqref{eq:specembedding}), our approach
yields a general algorithmic paradigm for finding many non-expanding sets.
For some $r \in \mathbb N$, do the following:
\begin{enumerate}
\item {\bf (Radial decomposition)}

Find disjoint subsets $S_1, S_2, \ldots, S_r \subseteq V$ using the values
$\{ F(v)/\|F(v)\| : v \in V \}$.

\item {\bf (Cheeger sweep)}

 For each $i=1,2,\ldots, r$,
\begin{quote}
Sort the vertices $S_i = \{ v_1, v_2, \ldots, v_{n_i} \}$
so that $$\|F(v_1)\| \geq \|F(v_2)\| \geq \cdots \geq \|F(v_{n_i})\|\,.$$
Output the least-expanding set among the $n_i-1$ sets of the form,
$$
\{v_1, v_2, \ldots, v_j \}
$$
for $1 \leq j \leq n_i-1$.
\end{quote}
\end{enumerate}

As discussed in the preceding section, each of our main theorems
is proved using an instantiation of this schema.
For instance, the proof of Theorem \ref{thm:kway} partitions using the radial
projection distance $d_F$.   The proof of \eqref{eq:planarintro} uses the
induced shortest-path metric $\hat d_F$.  And the proof of \eqref{eq:neartight}
uses $d_{F'}$ where $F' : V \to \mathbb R^{O(\log k)}$ is obtained
from random projection. The details of the scheme for equation \eqref{eq:neartight}
is provided in Section \ref{sec:algorithm}.
A practical algorithm might use $r$-means to cluster according to the
radial projection distance.

We remark that partitioning the normalized vectors as in step (i) is used
in the approach of \cite{NJW02}, but not in some other
methods of spectral partitioning (see \cite{VM03} for
alternatives).
Unlike \cite{NJW02}, our spectral partitioning algorithm
does not use directly the eigenvectors of the normalized Laplacian;
the vectors we use are multiplied by $D^{1/2}$ where
$D$ is the diagonal degree matrix (see Section \ref{sec:sptheory}).
In other words, we use the right eigenvectors of the
associated random walk matrix.
This is similar to \cite{ShiMalik}, except that they do not
normalize the spectral embedding as in our step (i).

\section{Preliminaries}

Let $G=(V,E,w)$ be a finite, undirected graph, with positive weights $w : E \to (0,\infty)$ on the edges.
For a pair of vertices $u,v \in V$, we sometimes write $w(u,v)$ for $w(\{u,v\})$.
For a subset of vertices $S \subseteq V$, we write $E(S,\overline{S}) := \{ \{u,v\} \in E : |\{u,v\} \cap S| = 1 \}$.
For a subset of edges $F \subseteq E$, we write $w(F) = \sum_{e \in F} w(e)$.
We use $x \sim y$ to denote $\{x,y\} \in E$.
We extend the weight to vertices by defining, for a single vertex $v\in V$, $w(v):=\sum_{u\sim v} w(u,v)$. We can think of $w(v)$ as the weighted degree of vertex $v$.
We will assume throughout that $w(v) > 0$ for every $v \in V$.  For $S \subseteq V$,
we write $w(S) = \sum_{v \in S} w(v)$.

Let $X$ be a set and $d : X \times X \to [0,\infty]$ is a symmetric non-negative function
which may take the value $\infty$.  We refer to $d$ as an {\em extended pseudo-metric on $X$}
if it satisfies the triangle inequality.  For a subset $S \subseteq X$, we write $\diam(S,d) \defeq \sup_{x,y \in S} d(x,y)$,
and for two sets $S,T \subseteq X$, we write $d(S,T) \defeq \inf_{x \in S, y \in T} d(x,y)$.
We also define the ball $B_{d}(x,R) \defeq \{ y \in X : d(x,y) \leq R \}$.

For two expressions $A$ and $B$, we write $A \lesssim B$ for $A \leq O(B)$ and $A \asymp B$
for the conjunction of $A \lesssim B$ and $A \gtrsim B$.

\subsection{Spectral theory of the weighted Laplacian}
\label{sec:sptheory}

We write $\ell^2(V, w)$ for the Hilbert space of functions $f : V \to \mathbb R$ with
inner product $$\langle f,g \rangle_{\ell^2(V,w)} \defeq \sum_{v \in V} w(v) f(v) g(v),$$
and norm $\|f\|_{\ell^2(V,w)}^2 = \langle f,f\rangle_{\ell^2(V,w)}$.
We reserve $\langle \cdot , \cdot \rangle$ and $\|\cdot\|$ for the standard
inner product and norm on $\mathbb R^k$, $k \in \mathbb N$ and $\ell^2(V)$.

We now discuss some operators on $\ell^2(V,w)$.  The adjacency operator
is defined by $A f(v) = \sum_{u \sim v} w(u,v) f(u)$, and the diagonal degree operator by
$D f(v) = w(v) f(v)$.  Then the {\em combinatorial Laplacian} is defined by
$L = D-A$, and the {\em normalized Laplacian} is given by
$$\mathcal L_G \defeq I - D^{-1/2} A D^{-1/2}.$$
Observe that for an unweighted, $d$-regular graph, we have $\mathcal L_G = \frac{1}{d} L$.

Now, if $g : V \to \mathbb R$ is a non-zero function and $f = D^{-1/2} g$, then
\begin{eqnarray*}
\frac{\langle g, \mathcal L_G \,g\rangle}{\langle g,g\rangle} &=& \frac{\langle g, D^{-1/2} L D^{-1/2} g\rangle}{\langle g,g\rangle} \\
&=&
\frac{\langle f, L f\rangle}{\langle D^{1/2} f, D^{1/2} f\rangle} \\
&=&
\frac{\displaystyle \sum_{u \sim v} w(u,v) |f(u)-f(v)|^2}{\displaystyle \sum_{v \in V} w(v) f(v)^2} \defeqb \mathcal R_G(f),
\end{eqnarray*}
where the latter value is referred to as the {\em Rayleigh quotient of $f$ (with respect to $G$)}.

In particular, one sees that $\mathcal L_G$ is a positive-definite operator with
eigenvalues $$0 = \lambda_1 \leq \lambda_2 \leq \cdots \leq \lambda_n \leq 2\,.$$
For a connected graph, the first eigenvalue corresponds
to the eigenfunctions $g = D^{1/2} f$, where $f$ is any non-zero constant function.
Furthermore, by standard variational principles,
\begin{eqnarray}
\lambda_k &=& \min_{g_1, \ldots, g_k \in \ell^2(V)} \max_{g \neq 0} \left\{ \frac{\langle g, \mathcal L_G\, g\rangle}{\langle g, g\rangle} : g \in \mathrm{span}\{g_1, \ldots,
g_k\}\right\} \nonumber \\
&=& \min_{f_1, \ldots, f_k \in \ell^2(V,w)} \max_{f \neq 0}  \left\{ \vphantom{\bigoplus} \mathcal R_G(f) : f \in \mathrm{span}\{f_1, \ldots, f_k\}\right\}, \label{eq:eigenvar}
\end{eqnarray}
where both minimums are over sets of $k$ non-zero orthogonal functions in the Hilbert spaces $\ell^2(V)$ and $\ell^2(V,w)$, respectively.
We refer to \cite{Chung97} for more background on the spectral theory
of the normalized Laplacian.

In particular, one can use \eqref{eq:eigenvar} to easily prove the left-hand side of \eqref{eq:kway} using the following
standard observation.

\begin{lemma}\label{lem:testfun}
Suppose $\psi_1, \psi_2, \ldots, \psi_k : V \to \mathbb R$ are disjointly-supported
functions with $\mathcal R_G(\psi_i)~\leq~c$ for each $i=1,2,\ldots,k$.  Then, $\lambda_k \leq 2c$.
\end{lemma}

\begin{proof}
Consider any $f = \sum_{i=1}^k \alpha_i \psi_i$.  Then, for any $u,v \in V$, we have
$$|f(u)-f(v)|^2 \leq 2 \sum_{i=1}^k \alpha_i^2 |\psi_i(u)-\psi_i(v)|^2\,,$$
using the fact the $\psi_i$'s
are disjointly supported.  Therefore,
\begin{eqnarray*}
\mathcal R_G(f) &=& \frac{\sum_{u \sim v} w(u,v) |f(u)-f(v)|^2}{\sum_{v \in V} w(v) f(v)^2} \\
&\leq & \frac{2 \sum_{i=1}^k \alpha_i^2 \sum_{u \sim v} w(u,v) |\psi_i(u)-\psi_i(v)|^2}{\sum_{i=1}^k \alpha_i^2 \sum_{v \in V} w(v) \psi_i(v)^2} \leq 2c\,.
\end{eqnarray*}
But now \eqref{eq:eigenvar} implies that $\lambda_k \leq \min_{f \in \mathrm{span}(\psi_1, \ldots, \psi_k)} \max_{f \neq 0} \mathcal R_G(f) \leq 2c$.
\end{proof}

Applying the preceding lemma with $\psi_i = \1_{S_i}$ as the indicator functions of disjoint sets $S_1, S_2, \ldots, S_k$ yields the
left-hand side of \eqref{eq:kway}, observing that $\phi_G(S_i) = \mathcal R_G(\1_{S_i})$.

\subsection{Cheeger's inequality with Dirichlet boundary conditions}

Given a subset $S \subseteq V$ by, we denote the {\em Dirichlet conductance} of $S$ by,
$$
\phi_G(S) \defeq \frac{w(E(S,\overline{S}))}{w(S)}\,.
$$
For convenience, we take $\phi_G(\emptyset) = \infty$.
If $\mathcal H$ is a Hilbert space, we extend the notion of Rayleigh quotients
to arbitrary maps $\psi : V \to \mathcal H$ via,
\begin{equation}\label{eq:extral}
\mathcal R_G(\psi)
 \defeq \frac{\displaystyle \sum_{u \sim v} w(u,v) \|\psi(u)-\psi(v)\|_{\mathcal H}^2}{\displaystyle \sum_{v \in V} w(v) \|\psi(v)\|_{\mathcal H}^2}\,.
\end{equation}
In what follows, we use $\supp(\psi) \defeq \{ v \in V : \psi(v) \neq 0 \}$.

Many variants of the following lemma are known; see, e.g. \cite{Chung96}.

\begin{lemma}\label{lem:cheeger}
For any $\psi : V \to \mathcal H$,
there exists a subset $S \subseteq \supp(\psi)$ with
$$
\phi_G(S) \leq \sqrt{2\, \mathcal R_G(\psi)}\,.
$$
\end{lemma}

\begin{proof}
Let $\|\cdot\| = \|\cdot\|_{\mathcal H}$.
We may assume that $\supp(\psi) \neq V$, else taking $S=V$ finishes the argument.
For $t \in [0,\infty)$, define a subset $S_t = \{ u \in V : \|\psi(u)\|^2 > t \}$.
Observe that for every $t \geq 0$, the inclusion $S_t \subseteq \supp(\psi)$ holds by construction.

Then we have the estimate,
$$
\int_0^\infty w(S_t)\,dt= \sum_{u \in V} w(u) \|\psi(u)\|^2,
$$
as well as,
\begin{eqnarray*}
\int_0^\infty  w(E(S_t,\overline{S_t}))\,dt &=&
\sum_{u \sim v} w(u,v) \left|\|\psi(u)\|^2 - \|\psi(v)\|^2\right| \\
&\leq & 
\sum_{u \sim v} w(u,v) \|\psi(u)-\psi(v)\| \cdot \|\psi(u)+\psi(v)\| \\
&\leq &
\sqrt{\sum_{u \sim v} w(u,v) \|\psi(u)-\psi(v)\|^2} \sqrt{\sum_{u \sim v} w(u,v) \|\psi(u)+\psi(v)\|^2} \\
&\leq &
\sqrt{\sum_{u \sim v} w(u,v) \|\psi(u)-\psi(v)\|^2} \sqrt{2 \sum_{u \in V} w(u) \|\psi(u)\|^2}\,.
\end{eqnarray*}
Combining these two inequalities yields,
$$
\frac{\int_0^\infty  w(E(S_t,\overline{S_t}))\,dt}{ \int_0^\infty w(S_t)\,dt} \leq \sqrt{2 \,\mathcal R_G(\psi)},
$$
implying there exists a $t \in [0,\infty]$ for which $S_t$ satisfies the statement
of the lemma.
\end{proof}

\subsection{Random partitions of metric spaces}
\label{sec:rps}

We now discuss some of the theory of random partitions
of metric spaces.
Let $(X,d)$ be a finite metric space.
We use $B(x,R) = \{ y \in X : d(x,y) \leq R \}$ to
denote the closed ball of radius $R$ about $x$.
We will write a partition $P$ of $X$ as a
function $P : X \to 2^X$ mapping a point $x \in X$
to the unique set in $P$ that contains $x$.

For $\Delta > 0$, we say that $P$ is {\em $\Delta$-bounded} if
$\diam(S) \leq \Delta$ for every $S \in P$.
We will also consider distributions over random partitions.
If $\mathcal P$ is a random partition of $X$, we say that $\mathcal P$
is $\Delta$-bounded if this property holds with probability one.

A random partition $\mathcal P$ is {\em $(\Delta, \alpha, \delta)$-padded} if
$\mathcal P$ is $\Delta$-bounded, and
for every $x \in X$, we have
$$
\pr[B(x, \Delta/\alpha) \subseteq \mathcal P(x)] \geq \delta.
$$
A random partition is {\em $(\Delta,L)$-Lipschitz} if $\mathcal P$ is $\Delta$-bounded,
and, for every pair $x,y \in X$, we have
$$
\pr[\mathcal P(x) \neq \mathcal P(y)] \leq L \cdot \frac{d(x,y)}{\Delta}\,.
$$

Here are some results that we will need.
The first theorem is known, more generally, for doubling spaces \cite{GKL03},
but here we only need its application to $\mathbb R^k$.
See also \cite[Lem 3.11]{LN05}.

\begin{theorem}\label{thm:rkpad}
If $X \subseteq \mathbb R^k$, then for every $\Delta > 0$ and $\delta > 0$, $X$ admits a $(\Delta, O(k/\delta), 1-\delta)$-padded random partition.
\end{theorem}

The next result is proved in \cite{CCGGP98}.
See also \cite[Lem 3.16]{LN05}.

\begin{theorem}\label{thm:RkLip}
If $X \subseteq \mathbb R^k$, then for every $\Delta > 0$, $X$ admits a $(\Delta, O(\sqrt{k}))$-Lipschitz random partition.
\end{theorem}

A partitioning theorem for excluded-minor graphs is presented in \cite{KPR93}, with
an improved quantitative dependence coming from \cite{FT03}.

\begin{theorem}\label{thm:planarpad}
If $X$ is the shortest-path metric on a graph excluding $K_h$ as a minor, then for every $\Delta > 0$ and $\delta > 0$,
$X$ admits a $(\Delta, O(h^2/\delta),1-\delta)$-padded random partition and a $(\Delta, O(h^2))$-Lipschitz
random partition.
\end{theorem}

Finally, for the special case of bounded-genus graphs, a better bound is known \cite{LS10}.

\begin{theorem}\label{thm:genuspad}
If $X$ is the shortest-path metric on a graph of genus $g$, for every $\Delta > 0$ and $\delta > 0$, $X$ admits a $(\Delta, O((\log g)/\delta), 1-\delta)$-padded random partition, and a
$(\Delta, O(\log g))$-Lipschitz random partition.
\end{theorem}

\section{Localizing eigenfunctions}

Let $G=(V,E,w)$ be a weighted graph.
In the present section, we show how to find,
for every $k \in \mathbb N$, disjointly supported functions $\psi_1, \psi_2, \ldots, \psi_k : V \to \mathbb R$
with $\mathcal R_G(\psi_i) \leq k^{O(1)} {\lambda_k}$, where $\lambda_k$
is the $k$th smallest eigenvalue of $\mathcal L_G$.

\subsection{The radial projection distance}

For $h \in \mathbb N$, consider a mapping $F : V \to \mathbb R^h$.
A central role will
be played by the {\em radial projection distance}, which is an extended pseudo-metric on $V$:  If $\|F(u)\|, \|F(v)\| > 0$, then
$$d_F(u,v) \defeq \left\|\frac{F(u)}{\|F(u)\|} - \frac{F(v)}{\|F(v)\|}\right\|\,.$$
Otherwise, if $F(u)=F(v)=0$, we put $d_F(u,v)\defeq 0$, else $d_F(u,v)\defeq \infty$.

In order to find many disjointly supported functions from a geometric
representation $F : V \to \mathbb R^h$, it should
be that the $\ell^2$ mass of $F$ is not too concentrated.
To this end, we say that $F$ is {\em $(\Delta ,\eta)$-spreading (with respect to $G$)} if,
for all subsets $S \subseteq V$, we have
$$
\diam(S, d_F) \leq \Delta \implies \sum_{u \in S} w(u) \|F(u)\|^2 \leq \eta \sum_{u \in V} w(u) \|F(u)\|^2\,.
$$

First, we record the following simple fact.

\begin{lemma}\label{lem:dF}
For any $F : V \to \mathbb R^h$,
and for all $u, v \in V$, we have $d_F(u,v) \|F(u)\| \leq 2\, \|F(u)-F(v)\|$.
\end{lemma}

\begin{proof}
For any non-zero vectors $x,y \in \mathbb R^k$, we have
$$
\|x\| \left\|\frac{x}{\|x\|} - \frac{y}{\|y\|}\right\|
=
\left\|x - \frac{\|x\|}{\|y\|} y\right\|
\leq
\|x-y\| + \left\|y - \frac{\|x\|}{\|y\|} y\right\|
\leq 2 \,\|x-y\|\,.
$$
\end{proof}

We now show that systems of $\ell^2(V,w)$-orthonormal functions give rise to spreading maps.

\begin{lemma}\label{lem:spread}
Suppose that $f_1, f_2, \ldots, f_k : V \to \mathbb R$ is an $\ell^2(V,w)$-orthonormal system and
that $F : V \to \mathbb R^k$ is given by
$F(v) = (f_1(v), f_2(v), \ldots, f_k(v))$.
Then, for every $\Delta > 0$, $F$ is $\left(\Delta, \frac{1}{k (1-\Delta^2)}\right)$-spreading
with respect to $G$.
\end{lemma}

\begin{proof}

Let $x \in \mathbb R^k$ be any unit vector, and
let $U : \mathbb R^k \to \ell^2(V,w)$ be defined by
$$(Ux)(v) \defeq \sum_{i=1}^k x_i \sqrt{w(v)} f_i(v)\,.$$
Observe that $(U^T U)_{i,j} = \langle f_i, f_j\rangle_{\ell^2(V,w)}$,
hence $U^TU = I$.  Thus,
\begin{equation}\label{eq:proj}
\sum_{v \in V} w(v) \langle x, F(v) \rangle^2
= \langle U x, U x\rangle = \langle x, U^T U x\rangle = 1.
\end{equation}

Now, let $S \subseteq V$ satisfy $\diam(S,d_F) \leq \Delta$.
Fix any $u \in S$ and use \eqref{eq:proj} to write,
$$
1 = \sum_{v \in V} w(v) \left\langle F(v), \frac{F(u)}{\|F(u)\|}\right\rangle^2
= \sum_{v \in V} w(v) \|F(v)\|^2 \left(1-\frac{d_F(u,v)^2}{2}\right)^2 \geq (1-\Delta^2) \sum_{v \in S} w(v) \|F(v)\|^2\,.
$$
The lemma now follows by noting that, $$\sum_{v \in V} w(v) \|F(v)\|^2 = \sum_{v \in V} \sum_{i=1}^k w(v) f_i(v)^2 = \sum_{v \in V} \sum_{i=1}^k \|f_i\|_{\ell^2(V,w)}^2 = k\,.$$
\end{proof}

\subsection{Smooth localization}
\label{sec:smoothlocal}

Given a map $F :V \to \mathbb R^h$
and a subset $S \subseteq V$, we now show how
to construct a function supported on a small-neighborhood $S$,
which retains the $\ell^2$ mass of $F$ on $S$, and which doesn't
stretch edges by too much.

For future applications,
it will be useful to consider the largest metric on $G$ which
agrees with $d_F$ on edges.  This is the induced shortest-path (extended pesudo-) metric on $G$,
where the length of an edge $\{u,v\} \in E$ is given by $d_F(u,v)$.
We will use the notation $\hat d_F$ for this metric.
Observe that $\hat d_F \geq d_F$ since $d_F$ is a pseudo-metric.
We will write $$N_{\e}(S,\hat d_F) \defeq \{ v \in V : \hat d_F(v,S) < \e \}$$ for the open $\e$-neighborhood of $S$
in the metric $\hat d_F$.

\begin{lemma}[Localization]\label{lem:bump}
For any $F : V \to \mathbb R^h$, the following holds.
For every subset $S \subseteq V$ and number $\e > 0$, there exists a mapping $\psi : V \to \mathbb R^h$ which
satisfies the following three properties:
\begin{enumerate}
\item $\psi|_S = F|_S$,
\item $\supp(\psi) \subseteq N_{\e}(S, \hat d_F)$, and
\item if $\{u,v\} \in E$, then $|\psi(u)-\psi(v)| \leq (1+\frac{2}{\e}) \|F(u)-F(v)\|$.
\end{enumerate}
\end{lemma}

\begin{proof}
First, define
$$
\theta(v) \defeq \max\left(0,1-\frac{\hat d_F(v,S)}{\e}\right)\,.
$$
In particular, observe that $\theta$ is $(1/\e)$-Lipschitz with respect to $\hat d_F$, so since $\hat d_F$ and $d_F$ agree on edges,
we have for every $\{u,v\} \in E$,
\begin{equation}\label{eq:dhatlip}
|\theta(u)-\theta(v)| \leq \frac{1}{\e} \,d_F(u,v)\,.
\end{equation}
Finally, set $\psi(v) \defeq \theta(v) F(v)$.

\medskip

Properties (i) and (ii) are immediate from the definition, thus we turn to property (iii).
Fix $\{u,v\} \in E$.
We have,
\begin{eqnarray*}
|\psi(u)-\psi(v)| &=& \left|\theta(u) F(u) - \theta(v) F(v)\vphantom{\bigoplus} \right| \\
&\leq & |\theta(v)| \cdot \left\|F(u) - F(v) \right\| + \|F(u)\| \cdot |\theta(u)-\theta(v)|\,.
\end{eqnarray*}
Since $\theta \leq 1$, the first term is at most $\|F(u)-F(v)\|$.  Now, using \eqref{eq:dhatlip}, and Lemma \ref{lem:dF}, we have
$$
  \|F(u)\| \cdot |\theta(u)-\theta(v)| \leq
\frac{1}{\e} \cdot \|F(u)\| \cdot d_F(u,v)
 \leq \frac{2}{\e}\cdot \|F(u)-F(v)\|,
$$
completing the proof of (iii).
\end{proof}

The preceding construction reduces the problem
of finding disjointly supported set functions
to finding separated regions in $(V,\hat d_F)$, each
of which contains a large fraction of the $\ell^2$ mass of $F$.

\begin{lemma}\label{lem:manybumps}
Let $F : V \to \mathbb R^h$ be given, and
suppose that for some $\beta,\delta > 0$ and $r \in \mathbb N$, there exist
$r$ disjoint subsets $T_1, T_2, \ldots, T_r \subseteq V$ such that
$\hat d_F(T_i, T_j) \geq \beta$ for $i \neq j$, and for every $i=1,2,\ldots, r$, we have
\begin{equation}\label{eq:mass}
\sum_{v \in T_i} w(v) \|F(v)\|^2 \geq \delta \sum_{v \in V} w(v) \|F(v)\|^2\,.
\end{equation}
Then there exist disjointly supported functions $\psi_1, \psi_2, \ldots, \psi_r : V \to \mathbb R$ such that
for $i=1,2,\ldots, r$, we have
\begin{equation}\label{eq:averaging}
\mathcal R_G(\psi_i) \leq \frac{2}{\delta(r-i+1)} \left(1+\frac{4}{\beta}\right)^2 \mathcal R_G(F)\,.
\end{equation}
\end{lemma}

\begin{proof}
For each $i \in [r]$, let $\psi_i : V \to \mathbb R^h$ be the result
of applying Lemma \ref{lem:bump} to the domain $T_i$ with parameter $\e = \beta/2$.
Since $\hat d_F(T_i, T_j) \geq \beta$ for $i\neq j$, property (ii) of Lemma \ref{lem:bump}
ensures that the functions $\{\psi_i\}_{i=1}^r$ are disjointly supported.

Additionally property (i) implies that for each $i \in [r]$,
$$
\sum_{v \in V} w(v) \|\psi_i(v)\|^2 \geq \sum_{v \in T_i} w(v) \|F(v)\|^2 \geq \delta \sum_{v \in V} w(v) \|F(v)\|^2,
$$
and by property (iii) of Lemma \ref{lem:bump}, and since the supports are disjoint,
$$
\sum_{u \sim v} \sum_{i=1}^r w(u,v) \|\psi_i(u)-\psi_i(v)\|^2 \leq 2 \left(1+\frac{4}{\beta}\right)^2 \sum_{u \sim v} w(u,v) \|F(u)-F(v)\|^2\,.
$$
In particular, if we reorder the maps so that $\mathcal R_G(\psi_1) \leq \mathcal R_G(\psi_2) \leq \cdots \leq \mathcal R_G(\psi_r)$,
then the preceding two inequalities imply \eqref{eq:averaging}.

These maps $\{\psi_i\}$ take values in $\mathbb R^h$, but it is easy to see that
for any $\psi : V \to \mathbb R^h$, there exists a coordinate $j \in \{1,2,\ldots,h\}$ such that
the map $\tilde \psi : V \to \mathbb R$ defined by $\tilde \psi(v) \defeq \psi(v)_j$ has $\mathcal R_G(\tilde \psi) \leq \mathcal R_G(\psi)$.
This follows from the general inequality $\frac{a_1+a_2+\cdots+a_k}{b_1+b_2+\cdots+b_k} \geq \min_i \frac{a_i}{b_i}$,
valid for all $a_1, \ldots, a_k, b_1, \ldots, b_k \geq 0$ with some $b_i > 0$.
\end{proof}

\subsection{Random partitioning}
\label{sec:randompart}

From Lemma \ref{lem:manybumps}, to find many disjointly supported functions with
small Rayleigh quotient, it suffices to partition $(V, \hat d_F)$ into well separated regions,
each of which contains a large fraction of the $\ell^2$ mass of $F$.
We will use a suitable distribution over random partitions and argue that
at least one partition in the support of the distribution is good for this purpose.

\begin{lemma}\label{lem:ksets}
Let $r, k \in \mathbb N$ be given with $k/2 \leq r \leq k$, and
suppose that the map $F : V \to \mathbb R^h$ is $(\Delta,\frac{1}{k} +\frac{k-r+1}{8kr})$-spreading for some $\Delta > 0$.
Suppose additionally there is a random partition $\mathcal P$ with the properties that
\begin{enumerate}
\item For every $S \in \mathcal P$, $\diam(S, d_F) \leq \Delta$, and
\item For every $v \in V$, $\pr[B_{\hat d_F}(v, \Delta/\alpha) \subseteq \mathcal P(v)] \geq 1 - \frac{k-r+1}{4r}$\,.
\end{enumerate}
Then there exist r disjoint subsets $T_1, T_2, \ldots, T_r \subseteq V$ such that for each $i \neq j$, we have $\hat d_F(T_i, T_j) \geq 2\Delta/\alpha$,
and for every $i =1,2,\ldots,k$,
$$
\sum_{v \in T_i} w(v) \|F(v)\|^2 \geq \frac{1}{2k} \sum_{v \in V} w(v) \|F(v)\|^2\,.
$$
\end{lemma}

\begin{proof}
For a subset $S \subseteq V$, define $$\tilde S \defeq \{x \in S : B_{\hat d_F}(x, \Delta/\alpha) \subseteq S\}\,.$$
Let $\mathcal M = \sum_{v \in V} w(v) \|F(v)\|^2 > 0.$
By linearity of expectation, there exists a partition $P$ such that
for every $S \in P$, $\diam(S,d_F) \leq \Delta$, and also
\begin{equation}\label{eq:haveleft}
\sum_{S \in P} \sum_{v \in \tilde S} w(v) \|F(v)\|^2 \geq \left(1-\frac{k-r+1}{4r}\right) \mathcal M\,.
\end{equation}

Furthermore, by the spreading property of $F$, we have, for each $S \in P$,
$$\sum_{x \in S} w(v) \|F(v)\|^2 \leq \frac{1}{k}\left(1 + \frac{k-r+1}{8r}\right) \mathcal M\,.$$
Therefore we may take disjoint unions of the sets $\{ \tilde S : S \in P\}$ to form at least $r$ disjoint sets
$T_1, T_2, \ldots, T_{r}$ with the property that for every $i=1,2,\ldots, r$, we have
$$
\sum_{v \in T_i} w(v) \|F(v)\|^2 \geq \frac{1}{2k} \cal M
$$
because the first $r-1$ pieces will have total mass at most
$$
\frac{r-1}{k} \left(1+\frac{k-r+1}{8r}\right){\cal M} \leq \left(1-\frac{k-r+1}{4r}-\frac{1}{2k}\right)\cal{E},
$$
for all $r \in [k/2,k]$,
leaving at least $\frac{\cal M}{2k}$ mass left over from \eqref{eq:haveleft}.
\end{proof}

We mention a representative corollary that follows from the conjunction of Lemmas \ref{lem:manybumps} and \ref{lem:ksets}.

\begin{corollary}
\label{cor:fewsets}
Let $k \in \mathbb N$ and $\delta \in (0,1)$ be given.
Suppose the map $F : V \to \mathbb R^h$ is $(\Delta, \frac{1}{k} + \frac{\delta}{48k})$-spreading
for some $\Delta \leq 1$, and
there is a random partition $\mathcal P$ with the properties that
\begin{enumerate}
\item For every $S \in \mathcal P$, $\diam(S, d_F) \leq \Delta$, and
\item For every $v \in V$, $\pr[B_{\hat d_F}(v, \Delta/\alpha) \subseteq \mathcal P(v)] \geq 1 - \frac{\delta}{24}$\,.
\end{enumerate}
Then there are at least $r \geq \lceil (1-\delta) k\rceil$
disjointly supported functions $\psi_1, \psi_2, \ldots, \psi_r : V \to \mathbb R$ such that
$$
\mathcal R_G(\psi_i) \lesssim \frac{\alpha^2}{\delta \Delta^2} \mathcal R_G(F)\,.
$$
\end{corollary}

\begin{proof}
In this case, we set $r=\lceil (1-\delta/2) k\rceil$ in our application of Lemma \ref{lem:ksets}.
After extracting at least $\lceil (1-\delta/2) k\rceil$ sets, we apply Lemma \ref{lem:manybumps},
but only take the first $r'=\lceil (1-\delta)k \rceil$ functions $\psi_1, \psi_2, \ldots, \psi_{r'}$.
\end{proof}

Note, in particular, that we can apply the preceding corollary with $\delta = \frac{1}{2k}$ to obtain $r=k$.

\remove{
If we only wish to extract $r \geq (1-\delta) k$ functions, the quantitative bounds improve significantly.

\begin{corollary}
\label{cor:fewsets}
For every $\delta \in (0,\frac12)$, the following holds.
Suppose that for some $k \in \mathbb N$, the map $F : V \to \mathbb R^h$ is $(\Delta, \frac{1}{k} + \frac{\delta}{48 k})$-spreading
for some $\Delta \leq 1$, and
there is a random partition $\mathcal P$ with the properties that
\begin{enumerate}
\item For every $S \in \mathcal P$, $\diam(S, d_F) \leq \Delta$, and
\item For every $v \in V$, $\pr[B_{\hat d_F}(v, \Delta/\alpha) \subseteq \mathcal P(v)] \geq 1 - \frac{\delta}{24}\,.$
\end{enumerate}
Then there are $r \geq (1-\delta)k $ disjointly supported functions $\psi_1, \psi_2, \ldots, \psi_r : V \to \mathbb R$ such that
$$
\mathcal R_G(\psi_i) \lesssim \frac{\alpha^2}{\delta \Delta^2} \mathcal R_G(F)\,.
$$
\end{corollary}

\begin{proof}
In this case, we set $r=\lceil (1-\delta/2) k\rceil$ in our application of Lemma \ref{lem:ksets}.
After extracting at least $\lceil (1-\delta/2) k\rceil$ sets, we apply Lemma \ref{lem:manybumps},
but only take the first $r'=\lceil (1-\delta)k \rceil$ functions $\psi_1, \psi_2, \ldots, \psi_{r'}$.
\end{proof}
}

\subsection{Higher-order Cheeger inequalities}
\label{sec:higherorder}

We now present some theorems applying our machinery to embeddings
which come from the eigenfunctions of $\mathcal L_G$.

\begin{theorem}\label{thm:ksets}
For any $\delta \in (0,1)$, and any weighted graph $G=(V,E,w)$, there exist
$r \geq \lceil (1-\delta)k\rceil$ disjointly supported functions $\psi_1, \psi_2, \ldots, \psi_r : V \to \mathbb R$ such that
\begin{equation}\label{eq:c1}
\mathcal R_G(\psi_i) \lesssim  \frac{k^{2}}{\delta^4} \lambda_k\,.
\end{equation}
where $\lambda_k$ is the $k$th smallest eigenvalue of $\mathcal L_G$.
If $G$ excludes $K_h$ as a minor, then the bound improves to
\begin{equation}\label{eq:c2}
\mathcal R_G(\psi_i) \lesssim  \frac{h^4}{\delta^4} \lambda_k,
\end{equation}
and if $G$ has genus at most $g \geq 1$, then one gets
\begin{equation}\label{eq:c3}
\mathcal R_G(\psi_i) \lesssim  \frac{\log^2(g+1)}{\delta^4} \lambda_k\,.
\end{equation}
\end{theorem}

\begin{proof}
Let $f_1, f_2, \ldots, f_k : V \to \mathbb R$ be an $\ell^2(V,w)$-orthonormal
system of eigenfunctions corresponding to the first $k$ eigenvalues of $\mathcal L_G$,
and define $F : V \to \mathbb R^k$ by $F(v)=(f_1(v), f_2(v), \ldots, f_k(v))$.

Choose $\Delta \asymp \sqrt{\delta}$ so that $(1-\Delta^2)^{-1} \leq 1+\frac{\delta}{48}$.
In this case, Lemma \ref{lem:spread} implies that $F$ is $(\Delta, \frac{1}{k} + \frac{\delta}{48k})$-spreading.
Now, for general graphs, since $d_F$ is Euclidean, we can use Theorem \ref{thm:rkpad} applied to $d_F$ to achieve $\alpha \asymp k/\delta$
in the assumptions of Corollary \ref{cor:fewsets}.  Observe that $\hat d_F \geq d_F$,
so that $B_{\hat d_F}(v, \Delta/\alpha) \subseteq B_{d_F}(v, \Delta/\alpha)$, meaning that
we can satisfy both conditions (i) and (ii), verifying \eqref{eq:c1}.

For \eqref{eq:c2} and \eqref{eq:c3}, we use Theorems \ref{thm:planarpad} and \ref{thm:genuspad}, respectively,
applied to the shortest-path metric $\hat d_F$.  Again, since $\hat d_F \geq d_F$, we have that $\diam(S,\hat d_F) \leq \Delta$ implies
$\diam(S,d_F) \leq \Delta$, so conditions (i) and (ii) are satisfied with $\alpha \asymp h^2/\delta$ and $\alpha \asymp \log(g+1)/\delta$,
respectively.
\end{proof}

We remark that in Section \ref{sec:dimreduce}, we will give an alternate bound of
$O(\delta^{-7} \log^2 k)\cdot \lambda_k$ for \eqref{eq:c1},
which is better for moderate values of $\delta$.

\medskip

Finally, we can use the preceding theorems in conjunction
with Lemma \ref{lem:cheeger} to produce many non-expanding sets.

\begin{theorem}(Non-expanding $k$-partition)\label{thm:kpartition}
For any weighted graph $G=(V,E,w)$, there exists a partition
$V = S_1 \cup S_2 \cup \cdots \cup S_k$ such that
\begin{equation*}
\mathcal \phi_G(S_i) \lesssim  k^{4} \sqrt{\lambda_k}\,.
\end{equation*}
where $\lambda_k$ is the $k$th smallest eigenvalue of $\mathcal L_G$.
If $G$ excludes $K_h$ as a minor, then the bound improves to
\begin{equation*}
\mathcal \phi_G(S_i) \lesssim  h^2 k^{3} \sqrt{\lambda_k}\,,
\end{equation*}
and if $G$ has genus at most $g \geq 1$, then one gets
\begin{equation*}
\mathcal \phi_G(S_i) \lesssim  \log (g+1) k^{3} \sqrt{\lambda_k}\,.
\end{equation*}
\end{theorem}

\begin{proof}
First apply Theorem \ref{thm:ksets} with $\delta = \frac{1}{2k}$
to find disjointly supported functions $\psi_1, \psi_2, \ldots, \psi_k : V \to \mathbb R$
satisfying \eqref{eq:c1}.  Now apply Lemma \ref{lem:cheeger} to find sets $S_1, S_2, \ldots, S_k$ with $S_i \subseteq \supp(\psi_i)$
and $\phi_G(S_i) \leq \sqrt{2 \mathcal R_G(\psi_i)}$ for each $i=1,2,\ldots,k$.

Now reorder the sets so that $w(S_1) \leq w(S_2) \leq \cdots \leq w(S_k)$, and replace $S_k$ with the larger set
$S_k' = V \setminus (S_1 \cup S_2 \cup \cdots \cup S_{k-1})$ so that
$V = S_1 \cup S_2 \cup \cdots \cup S_{k-1} \cup S'_k$ forms a partition.
One can now easily check that
$$
\phi_G(S'_k) = \frac{w(E(S'_k,\overline{S'_k}))}{w(S'_k)} \leq \frac{\sum_{i=1}^{k-1} w(E(S_i,\overline{S_i}))}{w(S'_k)} \leq k \cdot \max_{i=1}^k \phi_G(S_k) \lesssim k^4
\sqrt{\lambda_k}\,.
$$
A similar argument yields the other two bounds.
\end{proof}

Using Theorem \ref{thm:ksets} in conjunction with Lemma \ref{lem:cheeger} again yields the following.

\begin{theorem}\label{thm:fewpieces}
For every $\delta \in (0,1)$ and any weighted graph $G=(V,E,w)$, there exist $r \geq \lceil (1-\delta)k\rceil$ disjoint sets
$S_1, S_2, \ldots, S_r \subseteq V$ such that,
\begin{equation}\label{eq:tobeimproved}
\mathcal \phi_G(S_i) \lesssim  \frac{k}{\delta^2} \sqrt{\lambda_k}\,.
\end{equation}
where $\lambda_k$ is the $k$th smallest eigenvalue of $\mathcal L_G$.
If $G$ excludes $K_h$ as a minor, then the bound improves to
\begin{equation*}
\mathcal \phi_G(S_i) \lesssim  \frac{h^2}{\delta^2} \sqrt{\lambda_k}\,,
\end{equation*}
and if $G$ has genus at most $g \geq 1$, then one gets
\begin{equation*}
\mathcal \phi_G(S_i) \lesssim  \frac{\log (g+1)}{\delta^2} \sqrt{\lambda_k}\,.
\end{equation*}
\end{theorem}

We remark that the bound \eqref{eq:tobeimproved} will be improved,
in various ways, in Section \ref{sec:improved}.

\section{Improved quantitative bounds}
\label{sec:improved}

A main result of this section is the following theorem.

\begin{theorem}\label{thm:optfewsets}
Let $G=(V,E,w)$ be a weighted graph and let $k \in \{1,2,\ldots,n\}$ and $\delta \in (0,1)$ be given.
Suppose that
$f_1, f_2, \ldots, f_k : V \to \mathbb R$ forms an $\ell^2(V,w)$-orthonormal system.
Then there exist $r \geq \lceil (1-\delta) k\rceil$ disjoint sets $S_1, S_2, \ldots, S_r \subseteq V$
with
$$
\phi_G(S_i) \lesssim \frac{1}{\delta^3} \sqrt{\frac{\sum_{i=1}^k \sum_{u \sim v} w(u,v)(f_i(u)-f_i(v))^2}{\sum_{i=1}^k \sum_{v \in V} w(v) f_i(v)^2} \cdot \log k}\
$$
\end{theorem}

\begin{corollary}\label{cor:optfewsets}
For any weighted graph $G=(V,E,w)$, $k \in \{1,2,\ldots,n\}$, and $\delta \in (0,1)$,
there exist $r \geq \lceil (1-\delta)k\rceil$ disjoint sets $S_1, S_2, \ldots, S_r \subseteq V$
with
$$
\phi_G(S_i) \lesssim \frac{1}{\delta^3} \sqrt{\lambda_k \log k}\,,
$$
where $\lambda_k$ is the $k$th smallest eigenvalue of $\mathcal L_G$.
\end{corollary}

\subsection{Dimension reduction}
\label{sec:dimreduce}

One should observe that in Theorems \ref{thm:ksets} and \ref{thm:fewpieces}, the loss of $k^2$ in \eqref{eq:c1} and $k$ in
\eqref{eq:tobeimproved} comes
from the dimension of the eigenfunction embedding.
To achieve somewhat better bounds for general graphs,
we now show how to drastically reduce the dimension while preserving
the Rayleigh quotient and spreading properties.

Let $g_1, g_2, \ldots, g_h$ be i.i.d. $k$-dimensional Gaussians, and
consider the random mapping $\Gamma_{k,h} : \mathbb R^k \to \mathbb R^h$ defined by
$\Gamma_{k,h}(x) = h^{-1/2} (\langle g_1, x \rangle, \langle g_2, x \rangle, \ldots, \langle g_h, x \rangle)$.
Then we have the following basic estimates (see, e.g. \cite[Ch. 15]{Mat01} or \cite[Ch. 1]{LT11}).
For every $x \in \mathbb R^k$,
\begin{equation}\label{eq:prop1}
\E\left[\|\Gamma_{k,h}(x)\|^2\right] = \|x\|^2,
\end{equation}
and, for every $\delta \in (0,\frac12]$,
\begin{equation}\label{eq:prop2}
\pr\left[\vphantom{\bigoplus}\|\Gamma_{k,h}(x)\|^2 \notin [(1-\delta)\|x\|^2, (1+\delta)\|x\|^2]\right] \leq 2 e^{-\delta^2 h/12}\,,
\end{equation}
and for every $\lambda \geq 2$,
\begin{equation}\label{eq:tail}
\pr\left[\vphantom{\bigoplus} \|\Gamma_{k,h}(x)\|^2 \geq \lambda \|x\|^2\right] \leq e^{-\lambda h/12}\,.
\end{equation}

\begin{lemma}\label{lem:dimreduce}
Let $G=(V,E,w)$ be a weighted graph.
For every $k \in \mathbb N$, $\Delta \in [0,1]$,  and $\eta \geq 1/k$,
the following holds.
Suppose that $F : V \to \mathbb R^k$ is $(\Delta,\eta)$-spreading.
Then for some value $$h \lesssim \frac{1+\log(k)+\log\left(\tfrac{1}{\Delta}\right)}{\Delta^2}\,,$$ with
probability at least $1/2$, the map $\Gamma_{k,h}$ satisfies both of the following conditions:
\begin{enumerate}
\item ${\cal R}_G(\Gamma_{k,h} \circ F) \leq 8 \cdot \mathcal R_G(F)$, and
\item $\Gamma_{k,h} \circ F$ is $(\Delta/4, (1+\Delta)\eta)$-spreading with respect to $G$.
\end{enumerate}
\end{lemma}

\begin{proof}
Let $\delta = \Delta/16$.
We may assume that $k \geq 2$.
Choose $h \asymp  (1+\log{k}+\log(\frac{1}{\Delta}))/\Delta^2$ large enough such that $2e^{-\delta^2 h /12} \leq \delta^2 k^{-3}/128 $.
Let $\Gamma=\Gamma_{k,h}$.

First, observe that \eqref{eq:prop1} combined with Markov's inequality implies that the
following holds with probability at least $3/4$,
\begin{equation}\label{eq:prop1num}
\sum_{u\sim v} w(u,v) \|\Gamma(F(u)) - \Gamma(F(v))\|^2 \leq 4 \cdot \sum_{u\sim v} w(u,v) \|F(u) - F(v)\|^2\,.
\end{equation}

Now define,
$$U:=\{v\in V: \|\Gamma(F(v))\|^2 \in [(1-\delta)\|F(v)\|^2, (1+\delta)\|F(v)\|^2],$$
By \eqref{eq:prop2}, for each $v\in V$,
\begin{equation}\label{eq:pbound}
\pr\left[ v\notin U\right] \leq \delta k^{-3}/128\,.
\end{equation}
Next, we bound the amount of $\ell^2$ mass that falls outside of $U$.

 Therefore,
by Markov's inequality, with probability at least $31/32$, we have
\begin{equation}\label{eq:weightaway}
\sum_{v\notin U} w(v)\|F(v)\|^2 \leq \frac{\delta k^{-3}}{4} \sum_{v\in V} w(v)\|F(v)\|^2\,.
\end{equation}
In particular, with probability at least $31/32$, we have
\begin{equation}\label{eq:prop1denum}
\sum_{v\in V} w(v) \|\Gamma(F(v))\|^2 \geq (1-\delta) \sum_{v\in U} w(v) \|F(v)\|^2 \geq \left(1-2\delta \right) \sum_{v \in V} w(v) \|F(v)\|^2\,.
\end{equation}
Combining our estimates for \eqref{eq:prop1num} and \eqref{eq:prop1denum}, we conclude
that (i) holds with probability at least $23/32$.  Thus we can finish by showing that (ii)
holds with probability at least $25/32$.
We first consider property (ii) for subsets of $U$.

\begin{claim}\label{claim:one}
With probability at least $7/8$, the following holds:
Equation \eqref{eq:prop1denum} implies that,
for any subset $S \subseteq U$ with $\diam(S) \leq \Delta/4$, we have
$$
\sum_{v \in S} w(v) \|\Gamma(F(v))\|^2 \leq (1+6\delta)\eta \sum_{v \in V} w(v) \|\Gamma(F(v))\|^2\,.
$$
\end{claim}

\begin{proof}
For every $u,v \in V$, define the event,
$$ {\cal A}_{u,v} = \left\{d_{\Gamma(F)} (u,v) \in \left[ {d_F(u,v)} (1-\delta) - 2\delta, {d_F(u,v)} (1+\delta) + 2\delta \right] \right\}$$
and let $I_{u,v}$ be the random variable indicating that ${\cal A}_{u,v}$ does {\em not} occur.

We claim that for $u,v \in V$,
${\cal A}_{u,v}$ occurs if $u,v\in U$, and $$\left\|\Gamma\left( \frac{F(u)}{\|F(u)\|} - \frac{F(v)}{\|F(v)\|}\right) \right\|  \in [(1-\delta)d_F(u,v), (1+\delta)d_F(u,v)]\,.$$
To see this, observe that,
\begin{eqnarray*}
d_{\Gamma(F)}(u,v)&=&\left\| \frac{\Gamma(F(u))}{\|\Gamma(F(u))\|} - \frac{\Gamma(F(v))}{\|\Gamma(F(v))\|}\right\| \\
&\geq& \left\| \frac{\Gamma(F(u))}{\|F(u)\|} - \frac{\Gamma(F(v))}{\|F(v)\|}\right\| - \left\| \frac{\Gamma(F(u))}{\|\Gamma(F(u))\|}  - \frac{\Gamma(F(u))}{\|F(u)\|}\right\|
- \left\| \frac{\Gamma(F(v))}{\|\Gamma(F(v))\|} - \frac{\Gamma(F(v))}{\|F(v)\|}\right\| \\
&\geq & \left\|\Gamma\left( \frac{F(u)}{\|F(u)\|} - \frac{F(v)}{\|F(v)\|}\right) \right\| - 2\delta\\
& \geq& (1-\delta)d_F(u,v) - 2\delta,
\end{eqnarray*}
where we have used the fact that $\Gamma$ is a linear operator.
The other direction can be proved similarly.

\medskip

Therefore, by
 \eqref{eq:prop2}, and a union bound, for any $u,v\in V$, $\pr [ I _{u,v}] \leq 3 \delta k^{-3}/128$.
Let,
 $${\cal M}_I :=\sum_{u,v\in V} w(u) w(v) \| F(u)\|^2 \|F(v)\|^2 I_{u,v}\,.$$
 By linearity of expectation, and Markov's inequality, we conclude that
 \begin{equation}\label{eq:amarkov}
 \pr\left[{\cal M}_I \geq \frac{\delta}{4 k^3} \left(\sum_{v \in V} w(v) \|F(v)\|^2\right)^2\right] \leq \frac{1}{8}\,.
 \end{equation}

Now suppose there exists a subset $S \subseteq U$ with $\diam(S, d_{\Gamma(F)}) \leq \Delta/4$
and $$\sum_{v \in S} w(v) \|\Gamma(F(v))\|^2 \geq (1+6\delta) \eta \sum_{v \in V} w(v) \|\Gamma(F(v))\|^2\,.$$
Fix a vertex $u\in S$. Since for every $v \in S \setminus B_{d_F}(u, \Delta/2)$, we
have $d_F(u,v)\geq \Delta/2$,  $d_{\Gamma(F)}(u,v) \leq \Delta/4$, and recalling that
$\delta = \Delta/16$, it must be
that $I_{u,v}=1$.
On the other hand, we have
\begin{eqnarray*}
\sum_{v \in S \setminus B_{d_F}(u,\Delta/2)} w(v) \|F(v)\|^2 &\geq&
\sum_{v \in S} w(v) \|F(v)\|^2 - \sum_{v \in B_{d_F}(u, \Delta/2)} w(v) \|F(v)\|^2 \\
&\geq &
(1-\delta) \sum_{v \in S} w(v) \|\Gamma(F(v))\|^2 - \eta \sum_{v \in V} w(v) \|F(v)\|^2
\\
&\geq &
(1-\delta) (1+6\delta) \eta \sum_{v \in V} w(v) \|\Gamma(F(v))\|^2 - \eta \sum_{v \in V} w(v) \|F(v)\|^2 \\
&\overset{\eqref{eq:prop1denum}}{\geq} &
\left[(1-2\delta) (1-\delta) (1+6\delta)-1\right] \eta \sum_{v \in V} w(v) \|F(v)\|^2 \\
&\geq &
\delta \eta \sum_{v \in V} w(v) \|F(v)\|^2\,,
\end{eqnarray*}
where we have used the fact that $S \subseteq U$ and also $\diam(B_{d_F}(u,\Delta/2)) \leq \Delta$
and the fact that $F$ is $(\Delta,\eta)$-spreading.  In the final line, we have used $\delta \leq 1/16$.

Thus under our assumption on the existence of $S$ and again using $S \subseteq U$, we have
\begin{eqnarray*}
\mathcal M_I &\geq& \sum_{u \in S} w(u) \|F(u)\|^2 \sum_{v \in S \setminus B_{d_F}(u,\Delta/2)} w(v) \|F(v)\|^2 \\
&\geq &
\delta \eta \left(\sum_{v \in V} w(v) \|F(v)\|^2\right) \sum_{u \in S} w(u) \|F(u)\|^2 \\
 &\geq &
 \delta \eta \left(\sum_{v \in V} w(v) \|F(v)\|^2\right) (1-\delta) \sum_{u \in S} w(u) \|\Gamma(F(u))\|^2 \\
 &\geq &
 \delta(1-\delta) \eta \left(\sum_{v \in V} w(v) \|F(v)\|^2\right)(1+6\delta) \eta \left(\sum_{v \in V} w(v) \|\Gamma(F(v))\|^2 \right) \\
 &\overset{\eqref{eq:prop1denum}}{\geq} &
 \delta(1-\delta)(1-2\delta)(1+6\delta)\eta^2 \left(\sum_{v \in V} w(v) \|F(v)\|^2\right)^2 \\
 &\geq &
 \frac{\delta}{4k^3}  \left(\sum_{v \in V} w(v) \|F(v)\|^2\right)^2\,,
\end{eqnarray*}
where the last inequality follows from $\eta \geq 1/k$ and $\delta \leq 1/16$.
Combining this with \eqref{eq:amarkov} yields the claim.
\end{proof}

The preceding claim guarantees a spreading property for subsets $S \subseteq U$.
Finally, we need to handle points outside $U$.


\begin{claim}
With probability at least $15/16$, we have
$$
\sum_{v \notin U} w(v) \|\Gamma(F(v))\|^2 \leq \delta k^{-3} \sum_{v \in V} w(v) \|F(v)\|^2\,.
$$
\end{claim}

\begin{proof}
Let ${\cal D}_u$ be the event that $u \notin U$,
and let $H_u := \|\Gamma(F(u))\|^2 \1_{{\cal D}_u}$.  Then,
\begin{equation}
\label{eq:sumhevent} \E\left[ \sum_{u\notin U} w(u) \|\Gamma(F(u))\|^2 \right] = \sum_{u\in V} w(u) \E\left[H_u\right]\,.
\end{equation}

Now we can estimate,
\begin{equation}\label{eq:fbound}
\frac{\E[H_u]}{\|F(u)\|^2} \leq 2 \,\pr(\mathcal D_u) + \pr\left(\frac{\|\Gamma(F(u))\|^2}{\|F(u)\|^2} > 2\right) \cdot \E\left[\frac{\|\Gamma(F(u))\|^2}{\|F(u)\|^2} \,\Big|\,
\|\Gamma(F(u))\|^2 > 2\, \|F(u)\|^2 \right]\,.
\end{equation}
Using the inequality, valid for all non-negative $X$,
$$\pr(X > \lambda_0) \cdot \E[X \mid X > \lambda_0] \leq \int_{\lambda_0}^{\infty} \lambda \cdot \pr(X > \lambda)\, d\lambda\,,$$
we can bound the latter term in \eqref{eq:fbound} by,
\begin{eqnarray*}
\int_{2}^{\infty} \lambda \cdot \pr\left(\vphantom{\bigoplus} \|\Gamma(F(u))\|^2 > \lambda \|F(u)\|^2\right)\,d\lambda
\leq \int_{2}^{\infty} \lambda e^{-\lambda h/12}\,d\lambda
= \left(\frac{24}{h}+\frac{144}{h^2}\right) e^{-h/6} \leq \frac{\delta}{128 k^3}\,,
\end{eqnarray*}
where we have used \eqref{eq:tail} and the initial choice
of $h$ sufficiently large.

It follows from this, \eqref{eq:fbound}, and \eqref{eq:pbound}, that
$$
\E[H_u] \leq \frac{3 \delta}{128 k^3} \|F(u)\|^2\,.
$$

\remove{
\begin{eqnarray*}
\frac{\E[H_u]} &=& \pr [D_u] \E\left[\|\Gamma(F(u))\|^2 ~|~ D_u \right] \\
&\leq & \pr \left[ \|\Gamma(F(u))\|^2 \leq (1-\delta)\|F(u)\|^2\right] \|F(u)\|^2 \\
&& +\pr \left[ (1+\delta)\|F(u)\|^2 \leq \|\Gamma(F(u))\|^2 \right] \E \left[\|\Gamma(F(u))\|^2 ~|~  \|\Gamma(F(u))\|^2 \leq 2\|F(u)\|^2 \right] \\
&& +\sum_{i=1}^\infty \pr \left[ 2^i\|F(u)\|^2 \leq \|\Gamma(F(u))\|^2 \right] \E \left[ \|\Gamma(F(u))\|^2 ~|~  \|\Gamma(F(u))\|^2 \leq 2^{i+1}\|F(u)\|^2 \right]\\
&\leq & \frac{\delta}{128k^3} \left\{ \|F(u)\|^2 +\sum_{i=0}^\infty 2^{i+1} \|F(u)\|^2 k^{2^{i}-1}\right\} \leq \frac{\delta}{16k^3} \|F(u)\|^2

\end{eqnarray*}
}
Therefore, by Markov's inequality,
$$ \pr\left[ \sum_{v\notin U} w(v)\|\Gamma(F(v))\|^2 > \delta k^{-3}\sum_{v\in V}w(v) \|F(v)\|^2\right] \leq \frac{3}{128}\,,$$
completing the proof.
\end{proof}

To conclude the proof of the lemma, we need to verify that (ii) holds with probability at least $25/32$.
But observe that if \eqref{eq:prop1denum} holds, then
the  conclusion of the preceding claim is,
$$
\sum_{v \notin U} w(v) \|\Gamma(F(v))\|^2 \leq \delta k^{-3} \sum_{v \in V} w(v) \|F(v)\|^2 \leq 2 \delta k^{-3} \sum_{v \in V} w(v) \|\Gamma(F(v))\|^2
\leq \delta \eta \sum_{v \in V} w(v) \|\Gamma(F(v))\|^2\,.
$$
Combining this with Claim \ref{claim:one} shows that with probability at least $25/32$,
$\Gamma \circ F$ is $(\Delta/4, (1+7\delta)\eta)$-spreading,
completing the proof.
\end{proof}

As an application of the preceding lemma, observe that we can improve
\eqref{eq:c1} in Theorem \ref{thm:ksets} to the following bound, which is sometimes stronger,
using the essentially same proof, but first obtaining a spreading representation $F :V \to \mathbb R^{O(\delta^{-2} \log k)}$
using Lemma \ref{lem:dimreduce}.

\begin{theorem}\label{thm:betterfewsets}
For any weighted graph $G=(V,E,w)$ and $\delta > 0$ the following holds.
For every $k \in \mathbb N$,
there exist
$r \geq \lceil (1-\delta)k\rceil$ disjointly supported functions $\psi_1, \psi_2, \ldots, \psi_r : V \to \mathbb R$ such that
\begin{equation}\label{eq:f1}
\mathcal R_G(\psi_i) \lesssim  \delta^{-7} \log^2(k+1) \lambda_k\,.
\end{equation}
where $\lambda_k$ is the $k$th smallest eigenvalue of $\mathcal L_G$.
\end{theorem}

\begin{proof}
Let $f_1, f_2, \ldots, f_k : V \to \mathbb R$ be an $\ell^2(V,w)$-orthonormal
system of eigenfunctions corresponding to the first $k$ eigenvalues of $\mathcal L_G$,
and define $F : V \to \mathbb R^k$ by $F(v)=(f_1(v), f_2(v), \ldots, f_k(v))$.

We may clearly assume that $\delta \geq \frac{1}{2k}$.
Choose $\Delta \asymp \delta$ so that $(1-16\Delta^2)^{-1} (1+4\Delta) \leq 1+\frac{\delta}{48}$.
In this case, for some choice of
$$
h \lesssim \frac{1+\log(k)+\log\left(\tfrac{1}{\Delta}\right)}{\Delta^2} \lesssim \frac{O(\log k)}{\delta^2}\,,
$$
with probability at least $1/2$, $\Gamma_{k,h}$ satisfies the conclusions
of Lemma \ref{lem:dimreduce}.  Assume that $\Gamma : \mathbb R^k \to \mathbb R^h$
is some map satisfying these conclusions.

Then combining (ii) from Lemma \ref{lem:dimreduce} with Lemma \ref{lem:spread}, we see that $\Gamma \circ F : V \to \mathbb R^h$ is
$(\Delta, \frac{1}{k} + \frac{\delta}{48k})$-spreading.
Now we finish as in the proof of Theorem \ref{thm:ksets}, using the fact that $h = O(\delta^{-2} \log k)$.
\end{proof}

\subsection{A multi-way Cheeger inequality}
\label{sec:mwcheeger}

Note that Theorem \ref{thm:betterfewsets} combined with Lemma \ref{lem:cheeger} is still not strong enough to
prove Theorem \ref{thm:optfewsets}.  To do that, we need to combine Lemma \ref{lem:dimreduce}
with a strong Cheeger inequality for Lipschitz partitions.

Let $G=(V,E,w)$ be a weighted graph, and $F : V \to \mathbb R^h$.
Set $M = \max \{ \|F(v)\|^2 : v \in V \}$.
Let $\tau \in (0,M)$ be chosen uniformly at random, and for any subset $S \subseteq V$, define
$$\hat S = \{ v \in S : \|F(v)\|^2 \geq \tau \}\,.$$

\begin{lemma}\label{lem:mwcheeger}
For every $\Delta > 0$,
there exists a partition $V = S_1 \cup S_2 \cup \cdots \cup S_m$ such that
for every $i \in [m]$, $\diam(S_i, d_F) \leq \Delta$, and
\begin{equation}\label{eq:mwcheeger}
\frac{\E\left[w(E(\hat S_1,\overline{\hat S_1})) + w(E(\hat S_2,\overline{\hat S_2})) + \cdots + w(E(\hat S_m,\overline{\hat S_m}))\right]}{\E\left[w(\hat S_1) + \cdots + w(\hat
S_m)\right]} \lesssim
\frac{\sqrt{h}}{\Delta} \sqrt{{\cal R}_G(F)}\,.
\end{equation}
\end{lemma}

\begin{proof}
Since the statement of the lemma is homogeneous in $F$, we may assume that $M=1$.
By Theorem \ref{thm:RkLip}, there exists an $\Delta$-bounded random partition $\mathcal P$ satisfying,
for every $u,v \in V$,
\begin{equation}\label{eq:lipschitz}
\pr(\mathcal P(u) \neq \mathcal P(v)) \lesssim \frac{\sqrt{h}}{\Delta} \cdot d_F(u,v)\,.
\end{equation}
Let $\mathcal P = S_1 \cup S_2 \cup \cdots \cup S_m$, where we recall that $m$ is a random number.

First, observe that,
$
\E [ w(\hat S_i) ] = \sum_{v \in S_i} w(v)\|F(v)\|^2,
$
thus,
\begin{equation}\label{eq:denom}
\E\left[w(\hat S_1) + \cdots + w(\hat S_m)\right] =  \sum_{v \in V} w(v)\|F(v)\|^2\,.
\end{equation}

Next, if $\{u,v\} \in E$ with $\|F(u)\|^2 \leq \|F(v)\|^2$, then we have
\begin{eqnarray*}
&& \!\!\!\!\!\!\!\!\!\!\!\!\!\!\!\!\!\!\!\!\!\!\!\!
\pr \left[\{u,v\} \in E(\hat S_1,\overline{\hat S_1}) \cup \cdots \cup E(\hat S_m,\overline{\hat S_m})\right] \\
&& \leq \pr\left[\mathcal P(u) \neq \mathcal P(v)\right] \cdot \pr\left[\|F(u)\|^2 \geq \tau \textrm{ or } \|F(v)\|^2 \geq \tau \mid \mathcal P(u) \neq \mathcal P(v) \right] \\
&& \,\,\,\,\,\, + \,\, \pr\left[\vphantom{\bigoplus} \tau \in [\|F(u)\|^2, \|F(v)\|^2] \,\big|\, \mathcal P(u) = \mathcal P(v)\right] \\
&&
\lesssim
\frac{\sqrt{h}}{\Delta} \cdot  d_F(u,v) \left(\|F(u)\|^2 + \|F(v)\|^2\right) + \|F(v)\|^2 - \|F(u)\|^2 \\
&&
\leq
\left(\|F(u)\| + \|F(v)\|\right) \left(\frac{\sqrt{h}}{\Delta} \cdot d_F(u,v) (\|F(u)\| + \|F(v)\|) + \|F(v)\| - \|F(u)\| \right) \\
&& \leq
\frac{5 \sqrt{h}}{\Delta}  \left(\|F(u)\| + \|F(v)\|\right) \|F(u)-F(v)\|,
\end{eqnarray*}
where in the final line we have used Lemma \ref{lem:dF}.

Thus, we can use Cauchy-Schwarz to write,
\begin{eqnarray*}
\E\left[w(E(\hat S_1,\overline{\hat S_1})) + \cdots + w(E(\hat S_m,\overline{\hat S_m}))\right] &\lesssim &
\frac{\sqrt{h}}{\Delta} \sum_{u \sim v} w(u,v)(\|F(u)\|+\|F(v)\|) \|F(u)-F(v)\| \\
&\leq &
\frac{\sqrt{h}}{\Delta}
\sqrt{\sum_{u \sim v} w(u,v)(\|F(u)\|+\|F(v)\|)^2} \\
&& ~~\cdot~\sqrt{\sum_{u \sim v} w(u,v)\|F(u)-F(v)\|^2} \\
&\leq &
\frac{\sqrt{h}}{\Delta}
\sqrt{2 \sum_{v \in V} w(v) \|F(v)\|^2} \sqrt{\sum_{u \sim v} w(u,v)\|F(u)-F(v)\|^2}\,.
\end{eqnarray*}
Combining this with \eqref{eq:denom} yields,
$$
\frac{\E_{\mathcal P} \E\left[w(E(\hat S_1,\overline{\hat S_1})) +  \cdots + w(E(\hat S_m,\overline{\hat S_m}))\right]}{\E_{\mathcal P}\E\left[w(\hat S_1) + \cdots + w(\hat S_m)\right]} \lesssim
\frac{\sqrt{h}}{\Delta} \sqrt{\frac{\sum_{u \sim v}w(u,v) \|F(u)-F(v)\|^2}{\sum_{v \in V} w(v)\|F(v)\|^2}}\,,
$$
where we use $\E_{\mathcal P}$ to denote expectation over the random choice of $\mathcal P$.  In particular, there must exist
a single partition $P$ satisfying the statement of the lemma.
\end{proof}

We can use the preceding theorem to find many non-expanding sets, assuming that $F : V \to \mathbb R^h$ has
sufficiently good spreading properties.

\begin{lemma}\label{lem:mwcheegerapp}
Let $G=(V,E,w)$ be a weighted graph and let $k \in \mathbb N$ and $\delta \in (0,1)$ be given.
If the map $F : V \to \mathbb R^h$ is $(\Delta,\frac{1}{k}+\frac{\delta}{4k})$-spreading,
then there exist $r \geq \lceil (1-\delta)k\rceil$ disjoint sets $T^*_1, T^*_2, \ldots, T^*_r$, such that
$$
\phi_G(T^*_i) \lesssim \frac{\sqrt{h}}{\delta \Delta} \sqrt{\mathcal R_G(F)}\,.
$$
\end{lemma}

\begin{proof}
Since $\lceil (1-\delta)k\rceil \leq k$,
we may assume that
\begin{equation}\label{eq:deltabound}
\frac1{k} \leq \delta + \frac{\delta^2}{2}\,.
\end{equation}
Let $V = S_1 \cup S_2 \cup \cdots \cup S_m$ be the partition guaranteed by applying Lemma \ref{lem:mwcheeger} to the mapping $F : V \to \mathbb R^h$.
Set $\mathcal M \defeq \sum_{v \in V} w(v) \|F(v)\|^2\,.$
Since $F$ is $(\Delta,\frac{1}{k}+\frac{\delta}{4k})$-spreading and each $S_i$ satisfies $\diam(S_i, d_F) \leq \Delta$,
we can form $r' \geq \lceil (1-\delta/2)k\rceil$ sets $T_1, T_2, \ldots, T_{r'}$ by taking disjoint unions
of the sets $\{S_i\}$ so that for each $i=1,2,\ldots,r'$, we have
\begin{equation}\label{eq:Teq}
\frac{\mathcal M}{4k} \leq \sum_{v \in T_i} w(v) \|F(v)\|^2 \leq \frac{\mathcal M}{k} \left(1+\frac{\delta}{4}\right)\,.
\end{equation}

To see this, suppose we start with the family $\{S_i\}$ and iteratively merge the two sets for which $\sum_{v \in S_i} w(v) \|F(v)\|^2$ is smallest
subject to the constraint that no set has a sum which exceeds $\frac{\mathcal M}{k} \left(1+\frac{\delta}{4}\right)$.
At the end of this process, let $T_1, T_2, \ldots, T_{r'}$ represent the sets constructed that satisfy \eqref{eq:Teq}.
We will have
$$
\sum_{i=1}^{r'} \sum_{v \in T_i} w(v) \|F(v)\|^2 > \mathcal M \left(1-\frac{1}{4k}\right)\,.
$$
Therefore,
$$
r' > k\frac{1-\frac{1}{4k}}{1+\frac{\delta}{4}} \geq k\frac{1-\frac{\delta}{4}-\frac{\delta^2}{8}}{1+\frac{\delta}{4}} \geq k\left(1-\frac{\delta}{2}\right)\,,
$$
where in the second inequality we have used \eqref{eq:deltabound}.

In particular, $\E[w(\hat T_i)] = \sum_{v \in T_i} w(v) \|F(v)\|^2 \in [\frac{1}{4} \frac{\mathcal M}{k}, (1+\frac{\delta}{4}) \frac{\mathcal M}{k}]$.

\medskip

Order the sets so that $\E[w(E(\hat T_i,\overline{\hat T_i}))] \leq \E[w(E(\hat T_{i+1},\overline{\hat T_{i+1}}))]$ for $i=1,2,\ldots,r'-1$, and
let $r = \lceil (1-\delta)k \rceil$.  Then from \eqref{eq:mwcheeger}, it must be that each $i=1,2,\ldots, r$ satisfies
\begin{eqnarray*}
\E[w(E(\hat T_i,\overline{\hat T_i}))] &\lesssim& \frac{1}{\delta k} \E \left[\sum_{j=1}^m w(E(\hat S_j,\overline{\hat S_j}))\right] \\
&\lesssim& \frac{\sqrt{h}}{\delta k \cdot \Delta} \cdot \sqrt{\mathcal R_G(F)} \cdot \mathbb E\left[\sum_{j=1}^m w(\hat S_j)\right]
\lesssim \frac{\sqrt{h}}{\delta k \cdot \Delta} \cdot \sqrt{\mathcal R_G(F)} \cdot \mathcal M\,.
\end{eqnarray*}
But $\E[w(\hat T_i)] \asymp \mathcal M/k$ for each $i=1,2,\ldots, r$, showing that
$$
\frac{\E[w(E(\hat T_i,\overline{\hat T_i}))]}{\E[w(\hat T_i)]} \lesssim \frac{\sqrt{h}}{\delta \Delta} \cdot \sqrt{\mathcal R_G(F)}\,.
$$
\end{proof}

We can already use this to improve \eqref{eq:tobeimproved} in Theorem \ref{thm:fewpieces}.

\begin{theorem}\label{thm:fewpiecesbetter}
For every $\delta \in (0,1)$ and any weighted graph $G=(V,E,w)$, there exist $r \geq \lceil (1-\delta)k\rceil$ disjoint, non-empty sets
$S_1, S_2, \ldots, S_r \subseteq V$ such that,
\begin{equation}\label{eq:improvedbound}
\phi_G(S_i) \lesssim  \frac{\sqrt{k}}{\delta^{3/2}} \sqrt{\lambda_k}\,.
\end{equation}
where $\lambda_k$ is the $k$th smallest eigenvalue of $\mathcal L_G$.
\end{theorem}

\begin{proof}
Let $\Delta \asymp \sqrt{\delta}$ be such that $(1-\Delta^2)^{-1} \leq 1+\frac{\delta}{4}$.
If we take $F : V \to \mathbb R^k$ to be the embedding coming from the first $k$ eigenfunctions of $\mathcal L_G$, then
Lemma \ref{lem:spread} implies that $F$ is $(\Delta, \frac{1}{k} + \frac{\delta}{4k})$-spreading.
Now apply Lemma
\ref{lem:mwcheegerapp}.
\end{proof}

Observe that setting $\delta = \frac{1}{2k}$ in the preceding theorem yields Theorem \ref{thm:kway}.

\medskip

And now we can complete the proof of Theorem \ref{thm:optfewsets}.

\begin{proof}[Proof of Theorem \ref{thm:optfewsets}]
Let $F(v) = (f_1(v), f_2(v), \ldots, f_k(v))$.
Choose $\Delta \asymp \delta$ so that $(1-16\Delta^2)^{-1} (1+4\Delta) \leq 1+\frac{\delta}{4}$.
In this case, for some choice of
$$
h \lesssim \frac{1+\log(k)+\log\left(\tfrac{1}{\Delta}\right)}{\Delta^2} \lesssim \frac{O(\log k)}{\delta^2}\,,
$$
with probability at least $1/2$, $\Gamma_{k,h}$ satisfies the conclusions
of Lemma \ref{lem:dimreduce}.  Assume that $\Gamma : \mathbb R^k \to \mathbb R^h$
is some map satisfying these conclusions.

Then combining the conclusions of Lemma \ref{lem:dimreduce} with Lemma \ref{lem:spread}, we see that $F^* \defeq \Gamma$ is
$(\Delta, \frac{1}{k} + \frac{\delta}{4k})$-spreading, takes values in $\mathbb R^h$, and satisfies
$\mathcal R_G(F^*) \leq 8 \cdot \mathcal R_G(F)$.
Now applying Lemma \ref{lem:mwcheegerapp} yields the desired result.
\end{proof}

\subsection{Gaps in the spectrum}
\label{sec:gaps}

We now show that if there are significant gaps in the spectrum of $G$, one can
obtain a higher-order Cheeger inequality with no dependence on $k$.

\begin{theorem}
\label{thm:gaps}
There is a constant $c > 0$ such that for any weighted graph $G=(V,E,w)$ and $k \in \mathbb N$, the following holds.
Let $\delta \in (0,\frac13)$ be such that $\delta k$ is an integer.
If $\lambda_{(1+\delta)k} > c \frac{(\log k)^2}{\delta^9} \lambda_k$, then there are at least $r \geq (1-3\delta)k$
disjointly supported functions $\psi_1, \psi_2, \ldots, \psi_r : V \to \mathbb R$ such that
\begin{equation}
\mathcal R_G(\psi_i) \lesssim \frac{\lambda_k}{\delta^3}\,,
\end{equation}
where $\lambda_k$ is the $k$th smallest eigenvalue of $\mathcal L_G$.
\end{theorem}

\begin{proof}
Let $f_1, f_2, \ldots, f_k : V \to \mathbb R$ be an $\ell^2(V,w)$-orthonormal
system of eigenfunctions corresponding to the first $k$ eigenvalues of $\mathcal L_G$,
and define $F : V \to \mathbb R^k$ by $F(v)=(f_1(v), f_2(v), \ldots, f_k(v))$.
We may assume that $\delta \geq 1/k$.

Using Lemma \ref{lem:dimreduce} (as in the proof of Theorem \ref{thm:betterfewsets}), there is a map $\Lambda : V \to \mathbb R^{h}$
where $h = O(\frac{\log k}{\delta^2})$, and the following hold:
\begin{enumerate}
\item $\Lambda$ is $(\Delta,\eta)$-spreading for some $\Delta \asymp \delta$ and $\eta = \frac{1}{k} + \frac{\delta}{16k}$,
\item $\mathcal R_G(\Lambda) \leq 8\mathcal R_G(F) \leq 8\lambda_k$\,.
\end{enumerate}

Since the radial projection distance $d_{\Lambda}$ is Euclidean, we can use Theorem \ref{thm:rkpad} to achieve
a $(\Delta/4, \alpha, 1-\delta/16)$-padded random partition $\mathcal P$ of $(V,d_{\Lambda})$ with $\alpha \asymp \frac{h}{\delta} \asymp \frac{\log k}{\delta^3}$.
For a subset $S \subseteq V$, let $$\tilde S \defeq \{ v \in S : B_{d_{\Lambda}}(v,\Delta/(4\alpha)) \subseteq S \}\,.$$
Then by linearity of expectation applied to the random partition $\mathcal P$, there must exist a fixed partition $P$ of $V$
such that for every $S \in P$, we have $\diam(S,d_{\Lambda}) \leq \Delta/4$ and
\begin{equation}
\label{eq:largemass}
\sum_{S \in P} \mathcal M_{\Lambda}(\tilde S) \geq \left(1-\frac{\delta}{16}\right) \mathcal M_{\Lambda}(V)\,
\end{equation}
where we define $\mathcal M_{\Lambda}(S) \defeq \sum_{v \in S} w(v) \|\Lambda(v)\|^2$ for any $S \subseteq V$.

Order the sets of $P$ as $\tilde S_1, \tilde S_2, \ldots$ so that $\mathcal M_{\Lambda}(\tilde S_i) \leq \mathcal M_{\Lambda}(\tilde S_{i+1})$ for each $i$.
We consider two cases.

\medskip
\noindent
{\bf Case I:} $\mathcal M_{\Lambda}(\tilde S_{(1-2\delta) k}) > \eta \cdot \mathcal M_{\Lambda}(V)/2.$

\medskip

In this case it must be that for $1 \leq i,j \leq (1-2\delta)k$ and $i \neq j$, we have
$$
B_{d_{\Lambda}}(\tilde S_i, \Delta/4) \cap B_{d_{\Lambda}}(\tilde S_j, \Delta/4) =  \emptyset\,.
$$
Otherwise, one can put $S \defeq B_{d_{\Lambda}}(\tilde S_i, \Delta/4) \cup B_{d_{\Lambda}}(\tilde S_j, \Delta/4)$ so that $\diam(S,d_{\Lambda}) \leq \Delta$ but
$\mathcal M_{\Lambda}(S) > \eta \cdot \mathcal M_{\Lambda}(V)$, which contradicts that fact that $\Lambda$ is $(\Delta,\eta)$-spreading.

Now by applying Lemma \ref{lem:manybumps} to the $(\Delta/4)$-separated sets $\tilde S_1, \tilde S_2, \ldots, \tilde S_{(1-2\delta)k}$,
we obtain $r \geq (1-3\delta)k$ disjointly supported functions $\psi_1, \psi_2, \ldots, \psi_r : V \to \mathbb R$ such that for each $i$,
$$
\mathcal R_G(\psi_i) \lesssim \frac{1}{\eta \cdot \delta k \cdot \Delta^2} \mathcal R_G(\Lambda) \lesssim \frac{\lambda_k}{\delta^3}\,.
$$

\medskip
\noindent
{\bf Case II:} $\mathcal M_{\Lambda}(\tilde S_{(1-2\delta) k}) \leq \eta \cdot \mathcal M_{\Lambda}(V)/2.$

\medskip

Since $\Lambda$ is $(\Delta,\eta)$-spreading, for any $i \leq (1-2\delta)k$, we have
$$
\mathcal M_{\Lambda}(\tilde S_i) \leq \eta \cdot \mathcal M_{\Lambda}(V)\,.
$$
Thus we can take disjoint unions of the sets $\{\tilde S : S \in P\}$ to form at least $s$ disjoint sets
$T_1, T_2, \ldots, T_s$ with $s \geq \lceil (1+3\delta/2)k\rceil$ such that for each $i$,
$$
\mathcal M_{\Lambda}(T_i) \geq \frac{\mathcal M_{\Lambda}(V)}{8k}\,.
$$
This is because the first $r-1$ pieces will have total mass at most
$$
\sum_{i=1}^{r-1} \mathcal M_{\Lambda}(\tilde S_i) \leq (1-2\delta)k \cdot \eta \cdot \mathcal M_{\Lambda}(V) + \frac{7\delta k}{2} \max \left\{\frac{\eta}{2}, \frac{1}{2k}\right\} \mathcal M_{\Lambda}(V) \leq (1-\tfrac{\delta}{4})(1+\tfrac{\delta}{16}) \mathcal M_{\Lambda}(V)\,,
$$
leaving at least $\frac{\delta}{8} \mathcal M_{\Lambda}(V) \geq \frac{1}{8k} \mathcal M_{\Lambda}(V)$ left over from \eqref{eq:largemass}.

Now applying Lemma \ref{lem:manybumps} to the $\Delta/(4\alpha)$-separated sets $T_1, \ldots, T_s$,
we obtain $s' \geq (1+\delta)k$ disjointly supported functions
$\psi_1, \psi_2, \ldots, \psi_{s'} : V \to \mathbb R$ such that for each $i$,
$$
\mathcal R_G(\psi_i) \lesssim \frac{\alpha^2}{\delta \cdot \Delta^2} \mathcal R_G(\Lambda) \lesssim \frac{(\log k)^2}{\delta^9} \lambda_k\,.
$$
But now Lemma \ref{lem:testfun} implies that
$$
\lambda_{(1+\delta)k} \leq c' \frac{(\log k)^2}{\delta^9} \lambda_k\,,
$$
for some constant $c' > 0$, contradicting our initial assumption (for $c=c'$).
\end{proof}

Lemma \ref{lem:cheeger} immediately yields the following corollary.

\begin{corollary}
Under the assumptions of Theorem \ref{thm:gaps}, there are at least $r \geq (1-3\delta)k$ non-empty, disjoint sets $S_1, S_2, \ldots, S_r \subseteq V$
such that $\phi_G(S_i) \lesssim \sqrt{\lambda_k/\delta^3}$.
\end{corollary}

Let us conclude this section by describing the consequences of the above results for spectral clustering algorithms.
The proof of Theorem \ref{thm:gaps} aligns with the folklore belief that, in spectral clustering,
the number of clusters is best chosen based on a large gap in the spectrum of the underlying graph.
Additionally, the proof provides a justification for the use of the $k$-means heuristic.
Observe that in Case I (the only possible case under the assumptions of the theorem), the support
of each of the functions $\psi_i$ is a ball of radius at most $\Delta$ with respect to the metric
$d_{\Lambda}$.  In other words, the vertices are concentrated in $\asymp k$ balls of small radius
after the dimension reduction step.  It seems plausible that the $k$-means heuristic
could successfully locate a good partition of the vertices in such a scenario.

\subsection{Noisy hypercubes}
\label{sec:noisycube}

In the present section, we review examples for which Corollary \ref{cor:optfewsets} is tight.
For $k \in \mathbb N$ and $\e \in (0,1)$ let $H_{k,\e}=(V,E)$ be the ``noisy hypercube'' graph, where $V=\{0,1\}^k$, and for any $x,y\in V$ there
is an edge of weight $w(x,y)=\eps^{\|x-y\|_1}$.
We put $n = |V| = 2^k$.

\begin{theorem}
\label{thm:noisycube}
For any $1 \leq C < k$ and $k \in \mathbb N$, and $S\subseteq V$ with $|S| \leq C n/k$, we have
$$\phi_{H_{k,\e}}(S) \gtrsim \sqrt{\lambda_k \log{(k/C)}}\,,$$
where $\e = \frac{\log(2)}{\log (k/C)}$.
\end{theorem}

\begin{proof}
Let $H = H_{k,\e}$.
First, the weighted degree of every vertex is
$$
w(x) = \sum_{y \in V} \eps^{\|x-y\|_1} = (1+\eps)^k\,.
$$
Therefore,
if we define $F_i : V \to \mathbb R$ by $F_i(x) = (-1)^{x_i}$, then
$$
\mathcal R_{H_{k,\e}}(F_i) =
\frac{ \sum_{\{x,y\}} w(x,y) |F_i(x) - F_i(y)|^2 }{\sum_{x} w(x) F_i(x)^2} = \frac{2\eps n(1+\eps)^{k-1}}{n(1+\eps)^k} \leq 2\eps\,.
$$
Thus $\lambda_k(H) \leq 2\e$.
We will now show that for $|S| \leq Cn/k$, one has $\phi_H(S) \geq \frac12$, completing the proof of the theorem.

\medskip

To bound $\phi_H(\cdot)$, we need to recall some Fourier analysis.
 For $f,g: \{0,1\}^k \rightarrow \mathbb{R}$ define the  inner product:
$$ \langle f,g\rangle_{L^2(V)} \defeq \frac{1}{n} \sum_{x\in \{0,1\}^k} f(x)g(x). $$
Given $S\subseteq [k]$, the Walsh function $W_S: \{0,1\}^k \rightarrow \mathbb{R}$ is defined by $W_S(x) = (-1)^{\sum_{i\in S} x_i}$.
The Walsh functions form an orthonormal basis with respect to the above inner product. Therefore, any function $f:\{0,1\}^k \rightarrow \mathbb{R}$ has a unique representation as
$f=\sum_{S \subseteq [n]} \widehat{f}(S) W_S,$
where $\widehat{f}(S) \defeq \langle f,W_S\rangle_{L^2(V)}$.

For $\eta \in [0,1]$, the Bonami-Beckner operator $T_\eta$ is defined as
$$ T_\eta f := \sum_{S\subseteq [n]} \eta^{|S|} \widehat{f}(S) W_S. $$
The Bonami-Beckner inequality \cite{Bonami70,Beckner75} states that
\begin{equation}\label{eq:BB}
\sum_{S\subseteq [n]} \eta^{|S|} \widehat{f}(S)^2 = \|T_{\sqrt\eta} f\|^2_2 \leq \|f\|^2_{1+\eta} = \left\{\frac{1}{n} \sum_{x\in \{0,1\}^k} f(x)^{1+\eta}\right\}^{\frac{2}{1+\eta}}.
\end{equation}

Let $A$ be the normalized adjacency matrix of $H$,
i.e. $A_{xy} = \frac{\e^{|x \oplus y|}}{(1+\e)^k}\,.$
It follows  from an elementary calculation that $W_S$ is an eigenvector of $A$ with eigenvalue $(\frac{1-\eps}{1+\eps})^{|S|}$, i.e.
$$ AW_S = \left(\frac{1-\eps}{1+\eps}\right)^{|S|}W_S.$$

\def\one{{\bf 1}}
For $S\subseteq [n]$, let $\one_S$ be the indicator function of $S$. Therefore,
\begin{eqnarray*}
\langle \1_S, A \1_S\rangle_{L^2(V)}
= \sum_{T\subseteq [n]} \widehat{\one}_S(T)^2 \left(\frac{1-\eps}{1+\eps}\right)^{|T|} \leq \| \one_S \|^2_{\frac{2}{1+\eps}} = \left(\frac{|S|}{n}\right)^{1+\eps},
\end{eqnarray*}
where the one last inequality follows from \eqref{eq:BB}.

Now, observe that for any $S \subseteq V$, we have $$w(E(S,\overline{S})) = w(S) - w(E(S,S)) = w(S) - (1+\e)^k n \langle \1_S, A \1_S\rangle_{L^2(V)}\,$$
where we have written $E(S,S)$ for edges with both endpoints in $S$.

Hence, for any subset $S \subseteq V$ of size $|S|\leq Cn/k$, we have
$$ \phi_H(S) = \frac{w(E(S,\overline{S}))}{w(S)} =
\frac{|S| -  n \langle \one_S, A \one_S\rangle_{L^2(V)}}{|S|} \geq 1-\left(\frac{|S|}{n}\right)^\eps \geq 1-(k/C)^{-\eps} \geq \frac12\, $$
where the last inequality follows by the choice of $\eps = \log(2)/\log{(k/C)}$.
\end{proof}

\begin{remark}
The preceding theorem shows that even if we only want to find a set $S$ of size $n/\sqrt{k}$, then for values of
$k \leq O(\log n)$, we can still only achieve a bound of the form $\phi_H(S) \lesssim \sqrt{\lambda_k \log k}$.
The state of affairs for $k \gg \log n$ is a fascinating open question.
\end{remark}

\section{Conclusion}

\subsection{Description of our algorithm}
\label{sec:algorithm}

In Section \ref{sec:genalg}, we gave a generic outline of our spectral partitioning algorithm.
We remark that our instantiations of this algorithm are simple to describe.
As an example,
suppose we are given a weighted graph $G=(V,E,w)$
Let ${\cal L}_G = I-D^{-1/2}AD^{-1/2}$ be the normalized Laplacian matrix of $G$
where $I$ is the identity matrix, $A$ is the adjacency matrix and $D$ is the diagonal matrix of vertex degrees.
We want to find $k$ disjoint sets,
each of expansion $O(\sqrt{\lambda_{2k} \log k})$ where $\lambda_{2k}$ is the $2k^\textrm{th}$ smallest eigenvalue of ${\cal L}_G$ (recall Theorem
\ref{thm:fewer}).  We specify a complete randomized algorithm.

\medskip
\begin{enumerate}[i)]
\item (Spectral embedding) We start by computing   $2k$ orthonormal vectors $g_1,\ldots,g_{2k}$ (think of them as functions, $g_i:V\to\R$)
such that
$$ \frac{\sum_{i=1}^{2k} \sum_{u\sim v} w(u,v) |g_i(u)-g_i(v)|^2}{\sum_{i=1}^{2k} \sum_{v\in V} w(v) g_i(v)^2} \leq O(\lambda_{2k})$$
Let $f_i=D^{-1/2} g_i$, i.e., for each $v\in V$, $f_i(v)=g_i(v)/\sqrt{w(v)}$.
Define the spectral embedding $F : V \to \mathbb R^{2k}$,
 by $F(v) = (f_1(v), f_2(v), \ldots, f_{2k}(v))$. 
\item (Random Projection) For some $h = O(\log k)$,
we perform random projection into $\mathbb R^h$. Let $\Gamma_{2k,h} : \mathbb R^{2k} \to \mathbb R^h$
be the random linear map given by
$$
\Gamma_{2k,h}(x) = \frac{1}{\sqrt{h}} \left(\langle g_1, x\rangle, \ldots, \langle g_h, x \rangle\right),
$$
where $\{g_1, \ldots, g_h\}$ are i.i.d. standard Gaussians.
Define $F^* \defeq \Gamma_{2k,h} \circ F : V \to \mathbb R^h$ so that
for each $v\in V$, $F^*(v)=\frac{1}{\sqrt{h}} (\langle g_1, F(v)\rangle, \ldots, \langle g_h,F(v)\rangle ).$

\item (Random partitioning) For some $R = \Theta(1)$, we perform the random space
partitioning algorithm from \cite{CCGG98} as follows:
Let $\mathcal B$ denotes the closed Euclidean unit ball in $\mathbb R^h$.
Consider $V \subseteq \mathcal B$
by identifying each vertex with its image under the map $v \mapsto F^*(v)/\|F^*(v)\|$.
Choose i.i.d. sequence of points $\{x_1, x_2, \ldots\}$ in $\mathcal B$ (chosen according to the Lebesgue measure) and
 form a partition of $V$ into the sets
$$
V = \bigcup_{i=1}^{\infty} \left[\vphantom{\bigoplus} V \cap B(x_i, R) \setminus \left(B(x_1, R) \cup \cdots \cup B(x_{i-1}, R)\right)\right]
$$
Here, $B(x,R)$ represents the closed Euclidean ball of radius $R$ about $x$, and it is easy to see that
this induces a partition of $V$ in a finite number of steps with probability one.
In other words, we assign each vertex $v\in V$ to the first point $x_i$ such that
$$ \norm{ x_i - \frac{F^*(v)}{\norm{F^*(v)}}} \leq R.$$
Let $V = S_1 \cup S_2 \cup \cdots \cup S_m$ be this partition.

\item (Merging) For a subset $S \subseteq V$, let $\mathcal M(S) = \sum_{v \in S} w(v) \|F^*(v)\|^2$.
We sort the partition $\{S_1, S_2, \ldots, S_m\}$ in decreasing order
according to $\mathcal M(S_i)$.
Let $k' = \lceil \frac{3}{2} k \rceil$.
Then for each $i = k'+1, k'+2, \ldots, m$,
we iteratively
set $S_{\ell} \defeq S_{\ell} \cup S_i$ where $$\ell = \argmin \{ \mathcal M(S_j) : j \leq k \}\,.$$
(Intuitively, we form $k'$ sets from our total of $m \geq k'$ sets by
balancing the $\mathcal M(\cdot)$-value among them.)
At the end, we are left with a partition $V = S_1 \cup S_2 \cup \cdots \cup S_{k'}$ of $V$ into $k' \geq 3k/2$ sets.

\item (Cheeger Sweep) To complete the algorithm, for each $i=1,2,\ldots,k'$, we choose a value $\tau$ such that
$$
\hat S_i = \{ v \in S_i : \|F^*(v)\|^2 \geq \tau \}
$$
has the least expansion.
We then output $k$ of the sets $\hat S_1, \hat S_2, \ldots, \hat S_{k'}$
that have the smallest expansion.
\end{enumerate}

We emphasize that one can run the above algorithm  using any set of orthonormal
vectors with small Rayleigh quotient.  One can employ the recent developments on fast Laplacian solvers
to find such vectors in near-linear time \cite{ST04,KMP11,KOSZ13,Vishnoi13}.
Given orthonormal vectors $g_1,\ldots,g_{2k}$, the above algorithms runs in time $O(n\cdot \poly(k))$.
In particular every step except random partitioning runs in nearly linear time, and the random
partitioning step runs in time $O(n\cdot 2^h)$.

\subsection{Future directions}

The preceding algorithm suggests some natural questions.  First, does dimension reduction
help to improve the quality of clusterings in practice?  For instance, if one runs the $k$-means
algorithm (as in \cite{NJW02}) on the randomly projected points, does it yield better results?
Another interesting question is whether, at least in certain circumstances,
the quality of the $k$-means clustering can be rigorously analyzed when
used in place of our random geometric partitioning.

\medskip

It would be interesting to find the right asymptotic dependence on $k$ in  Theorem \ref{thm:kway}.
Recall that in Theorems \ref{thm:fewer} and \ref{thm:noisycube},
we showed that if one is interested in finding, say, $k/2$ disjoint non-expanding sets,
then the right dependence on $k$ is $\Theta(\sqrt{\log k})$.

One might hope that it is possible to achieve $\rho_G(k) \leq (\log(k))^{O(1)} \sqrt{\lambda_k}$.
Such a bound is impossible if we instead try to find a $k$-{\em partitioning} of our graph. There are simple family of graphs where the sparsity of the
best $k$-partitioning has
a polynomial dependence on $k$ \cite{LRTV12}.

\bibliographystyle{alpha}
\bibliography{references,trees,hs}

\end{document}